\documentclass[12pt,reqno]{amsart}

\usepackage{amsfonts, amsthm, amsmath, amssymb, amscd}
\allowdisplaybreaks[4]

\usepackage{rotating}

\usepackage{subeqnarray,eepicemu,url,cite,bm}
\usepackage{xcolor}

\usepackage{tikz}
\usetikzlibrary{shapes.misc}

\usepackage{stmaryrd}
\usepackage{multirow}

\usepackage{gensymb}
\usepackage{bm}
\usepackage{extarrows}
\usepackage[latin2]{inputenc}

\usepackage{t1enc}

\usepackage[mathscr]{eucal}

\usepackage{indentfirst}

\usepackage{graphicx}

\usepackage{graphics}

\usepackage{pict2e}

\usepackage{mathrsfs}

\usepackage{enumerate}

\usetikzlibrary{matrix,arrows}
\usetikzlibrary{positioning}
\usetikzlibrary{fit}
\usetikzlibrary{patterns}

\usepackage{pgfplots}

\usepackage{CJK}
\usepackage[pagebackref]{hyperref}
\hypersetup{colorlinks=true}
\usepackage{cite}
\usepackage{color}
\usepackage{epic}
\usepackage{hyperref} 
\numberwithin{equation}{section}
\def\blue{\textcolor{blue}}
\def\red{\textcolor{red}}
\theoremstyle{plain}
\usetikzlibrary{shapes}

\tikzstyle{pathdefault}=[draw, line width=1, solid, color=black]
\tikzstyle{nodedefault}=[circle, inner sep=1.5, fill=black]
\tikzstyle{empty}=[]
\tikzstyle{nodeellipsis}=[circle, inner sep=0.5, fill=black]
\tikzstyle{pathcolor1}=[draw, line width=1.3, densely dashed, color=red]
\tikzstyle{pathcolor2}=[draw, line width=1.6, densely dotted, color=blue]
\tikzstyle{pathcolorlight}=[draw, line width=1, dotted, color=lightgray]

\tikzstyle{arbpathcolor0}=[line width=1, dashdotted, color=black]
\tikzstyle{arbpathcolor1}=[line width=1, densely dashed, color=red]
\tikzstyle{arbpathdefault}=[line width=1, densely dotted, color=blue]

\newcounter{id}
\newcommand{\drawlinedotswithstyle}[4]{
 \def\x{{#3}}
 \def\y{{#4}}
 \tikzstyle{thispathstyle}=[#1]
 \tikzstyle{thisnodestyle}=[#2]
 \setcounter{id}{-1} 
 \foreach \j in {#3}{\stepcounter{id}} 
 \foreach \i in {1,...,\the\value{id}}{  
  \path[thispathstyle] (\x[\i],\y[\i]) --(\x[\i-1],\y[\i-1]); 
 }
 \foreach \i in {1,...,\the\value{id}}{  
  \node[thisnodestyle] at (\x[\i],\y[\i]) {}; 
 }
 \node[thisnodestyle] at (\x[0],\y[0]) {}; 
}

\definecolor{mhcblue}{HTML}{0077CC} 
\definecolor{davidsonred}{HTML}{AC1A2F} 

\definecolor{green}{RGB}{0, 180, 0}
\definecolor{yellow}{RGB}{180, 180, 0}

\tikzset{cross/.style={cross out, draw=black, minimum size=2*(#1-\pgflinewidth), inner sep=0pt, outer sep=0pt},
cross/.default={3pt}}

\newtheorem{theorem}{Theorem}[section]

\newtheorem{lemma}[theorem]{Lemma}

\newtheorem{corollary}[theorem]{Corollary}

\newtheorem{proposition}[theorem]{Proposition}

\theoremstyle{definition}

\newtheorem{example}[theorem]{Example}

\newtheorem{conjecture}[theorem]{Conjecture}
\newtheorem{conj}[theorem]{Conjecture}

\newtheorem{remark}[theorem]{Remark}

\newtheorem{?}[theorem]{Problem}
\theoremstyle{remark}
\newtheorem*{rem}{Remark}

\newcommand{\be}{\begin{equation}}
\newcommand{\ee}{\end{equation}}
\newcommand{\Sym}{{\mathfrak{S}}}

\newcommand{\dbrac}[1]{{\llbracket#1\rrbracket}} 	
\newcommand{\boks}[2]{({#1, #2})}   

\def\Z{\mathbb{Z}}
\def\Sym{\mathfrak{S}}

\def\exc{\operatorname{exc}}
\def\des{\operatorname{des}}
\def\asc{\operatorname{asc}}
\def\dasc{\operatorname{dasc}}
\def\ddes{\operatorname{ddes}}
\def\valley{\mathrm{valley}}

\def\fix{\operatorname{fix}}
\def\fmax{\operatorname{fmax}}

\def\nest{\operatorname{nest}}

\def\dd{\operatorname{dd}}

\def\valley{\operatorname{val}}
\def\nest{\operatorname{nest}}

\def\drop{\operatorname{drop}}

\def\ddscent{\mathrm{ddes}}

\def\cval{\operatorname{cval}}
\def\cpeak{\operatorname{cpeak}}
\def\peak{\operatorname{peak}}

\def\cyc{\operatorname{cyc}}
\def\Orb{\mathrm{Orb}}
\def\fmax{\operatorname{fmax}}
\def\fmin{\operatorname{fmin}}
\def\lrm{\operatorname{lrm}}

\def\Arda{\mathrm{Arda}}
\def\cab{31\text{-}2}
\def\bca{2\text{-}31}

\def\SDE{\mathfrak D}

\def\des{\operatorname{des}}         
\def\asc{\operatorname{asc}}
\def\Des{\operatorname{Des}}         

\def\eareccpeakbis{\operatorname{eareccpeak}}
\def\ear{\operatorname{ear}}
\def\eareccpeak{\operatorname{ear}}

\def\valley{\operatorname{val}}

\def\rec{\operatorname{rec}}

\def\am{0cm} \def\Am{0.5cm}
\def\bm{1cm} \def\Bm{1.5cm}
\def\cm{2cm} \def\Cm{2.5cm}
\def\dm{3cm} 
\def\sizePoint{4pt}
\newcommand{\point}[2]{\fill (canvas cs:x=#1,y=#2) circle (\sizePoint); }

\DeclareMathOperator\Cval{Cval}
\DeclareMathOperator\Cdfall{Cdfall}
\DeclareMathOperator\Cdrise{Cdrise}
\DeclareMathOperator\Cpeak{Cpeak}
\DeclareMathOperator\Fix{Fix}
\DeclareMathOperator\Pex{Pex}

\DeclareMathOperator\Fmax{Fmax}
\DeclareMathOperator\cdfall{cdfall}
\DeclareMathOperator\cdrise{cdrise}
\DeclareMathOperator\Valley{Val}
\DeclareMathOperator\Ddes{Ddes}
\DeclareMathOperator\Dasc{Dasc}
\DeclareMathOperator\Peak{Peak}

\def\pex{\operatorname{pex}}
\def\pdrop{\operatorname{pdrop}}
\def\Pdrop{\operatorname{Pdrop}}
\def\pcyc{\operatorname{pcyc}}

\def\ucross{\operatorname{ucross}}
\def\lnest{\operatorname{lnest}}
\def\stan{\operatorname{stan}}

\def\pval{\mathrm{pval}}
\def\ppk{\mathrm{ppeak}}
\newcommand{\sfa}{{{\sf a}}}
\newcommand{\sfb}{{{\sf b}}}
\newcommand{\sfc}{{{\sf c}}}
\newcommand{\sfd}{{{\sf d}}}
\newcommand{\sfe}{{{\sf e}}}

\newcommand{\bsfa}{{\mbox{\textsf{\textbf{a}}}}}
\newcommand{\bsfb}{{\mbox{\textsf{\textbf{b}}}}}
\newcommand{\bsfc}{{\mbox{\textsf{\textbf{c}}}}}
\newcommand{\bsfd}{{\mbox{\textsf{\textbf{d}}}}}
\newcommand{\bsfe}{{\mbox{\textsf{\textbf{e}}}}}

\def\unest{\mathrm{unest}}
\def\lnest{\mathrm{lnest}}
\def\lev{\mathrm{lev}}
\newcommand{\cross}{\operatorname{cross}}
\newcommand{\icross}{\operatorname{icross}}

\newcommand{\lcross}{\operatorname{lcross}}

\newcommand*{\Scale}[2][4]{\scalebox{#1}{$#2$}}
\newcommand{\eqdef}{\stackrel{\rm def}{=}}

\begin{document}
\title[Equidistributions around  special kinds of descents and excedances]
{Equidistributions around  special kinds of descents and excedances}
\author[B. Han]{Bin Han}
\address[Bin Han]{Department of Mathematics, Bar-Ilan University, Ramat-Gan 52900, Israel.}
\email{han@math.biu.ac.il, han.combin@hotmail.com}

\author[J. Mao]{Jianxi Mao}
\address[Jianxi Mao]{School of mathematical sciences, Dalian University of Technology,  Dalian 116024, P.R China.}
\email{maojianxi@hotmail.com}

\author[J. Zeng]{Jiang Zeng}
\address[Jiang Zeng]{Univ Lyon, Universit\'e Claude Bernard Lyon 1, CNRS UMR 5208, Institut Camille Jordan, 43 blvd. du 11 novembre 1918, F-69622 Villeurbanne cedex, France}
\email{zeng@math.univ-lyon1.fr}

\thanks{BH was supported by the Israel Science Foundation, grant no. 1970/18}

\date{\today}

\begin{abstract}
We consider a sequence of four variable  polynomials by refining  Stieltjes' continued fraction for Eulerian polynomials. 
Using combinatorial theory of  Jacobi-type continued fractions and bijections we derive various combinatorial interpretations in terms of permutation statistics for these polynomials, which include  
special kinds of descents and excedances in a recent paper of
Baril and Kirgizov.
  As a by-product, we derive   several equidistribution results for permutation  statistics, which enables us to  
 confirm and strengthen a recent conjecture of Vajnovszki and also to obtain
 several compagnion permutation statistics for 
 two bistatistics in a conjecture of Baril and Kirgizov.
\end{abstract}

\subjclass[2010]{05A05, 05A15, 05A19}

\keywords{Eulerian polynomials, bijection, permutation statistic, equidistribution, cycle,  continued fraction, descent, excedance, drop,  derangement, gamma-positivity}

\maketitle


\section{Introduction}
It is well-known~\cite{FS70, St97, Pet15}  that
the statistics
"$\des$" and "$\exc$" are equidistributed over
permutations  of $[n]:=\{1, \ldots, n\}$, their common generating function being  the
Eulerian polynomials $A_n(t)$, i.e., 
\begin{align*}
A_n(t)=\sum_{\sigma\in \Sym_n}t^{\des  \sigma}=\sum_{\sigma\in \Sym_n}t^{\exc \sigma},
\end{align*}
which  satisfy  the identity
$$
\frac{A_n(t)}{(1-t)^{n+1}}=\sum_{r=0}^\infty t^r(r+1)^n.
$$
Since MacMahon's pioneering work~\cite{Mac15} various combinatorial  variants and refinements of  Eulerian polynomials have 
appeared, see  \cite{BS20, FHL20, HMZ20, MZ21, SZ20,MMYY21} for some recent papers.  

In a recent paper~\cite{BK21} Baril and Kirgizov
considered  some special  descents, excedances and cycles of permutations, that we recall in the following.
For a permutation
$\sigma:=\sigma(1)\sigma(2)\cdots\sigma(n)$  of $1\ldots n$,
an index  $i\in [1,\, n-1]$ is called a
\begin{itemize}
\item   \textit{descent} (resp. \textit{excedance}) 
 if $\sigma(i)>\sigma(i+1)$ (resp. $\sigma(i)>i$);
 \item  \textit{descent of type 2} 
 if  $i$ is a descent and $\sigma(j)<\sigma(i)$ for $j<i$;
\item  \textit{pure excedance}  if
$i$ is an excedance and $\sigma(j)\notin [i,\sigma(i)]$ for $ j< i$;
\end{itemize}
and an index $i\in [2,\,n]$ is called a 
\begin{itemize}
\item \textit{drop}  if
$i>\sigma(i)$;
\item  \textit{pure drop}  if
$i$ is a drop and $\sigma(j)\notin [\sigma(i), i]$ for $ j> i$.
 \end{itemize}
Let $\des\,\sigma$
(resp.  $\exc\,\sigma$, $\drop\sigma$,  $\des_2\,\sigma$, $\pex\,\sigma$ and  $\pdrop\, \sigma$) denote the number of descents (resp. excedances, drops,  descents of type 2,  pure excedances and 
pure drops)
of $\sigma$.  Identifying $\sigma$ with the bijection $i\mapsto \sigma(i)$ on $[n]$ we can decompose $\sigma$ into disjoint cycles
$(i, \sigma(i), \ldots, \sigma^\ell(i))$ with $\sigma^{\ell+1}(i)=i$ and $i\in [n]$. A cycle with
$\ell=1$ is called a \emph{fixed point} of $\sigma$.
Let $\cyc\,\sigma$ (resp. $\fix\,\sigma$) denote the number of cycles (resp. fixed points) of $\sigma$. The
 number of   non trivial cycles of $\sigma$
 \cite[A136394]{Sloa} is defined by
 \begin{align}\label{eq:defpcyc}
\pcyc\,\sigma= \cyc\,\sigma - \fix\,\sigma. 
\end{align}
For example, if $n=8$ and $\sigma=2\,3\,1\,4\,6\,8\,7\,5$,
the descent indexes of type 2 are $\{2, 6\}$;
the pure excedance indexes are $\{1, 5\}$ and the pure drop indexes are 
$\{3, 8\}$. Thus $\des_2\,\sigma=2$,  $\pex\,\sigma=2$, and $\pdrop\,\sigma=2$.
Factorizing   $\sigma$ as product of  disjoint cycles $(1\,2\,3)(4)(5\,6\,8)(7)$, we derive  $\cyc\,\sigma=4$, $\fix\,\sigma=2$, and $\pcyc\,\sigma=2$.


A \emph{mesh pattern} of length $k$ is a pair $(\tau,R)$, where $\tau$ is a permutation of length $k$ and $R$ is a subset of $\dbrac{0,k} \times \dbrac{0,k}$
with  $\dbrac{0,k}=\{0,1, \ldots, k\}$.
Let $\boks{i}{j}$ denote the box whose corners have coordinates $(i,j), (i,j+1),
(i+1,j+1)$, and $(i+1,j)$.
Note that a
 descent of type $2$ can be viewed as an occurrence of the mesh pattern
$(21,L_1)$ where $L_1=\{1\}\times[0,2]\cup \{(0,2)\}$.
By abuse of notation, 
we use  $\des_2$ to denote the
mesh pattern   corresponding to 
an occurrence of descent of type $2$  in Figure~\ref{fig1}. Similarly, we use
$\pex$ (resp. $\pdrop$) to denote 
 an occurence of pure excedance in Figure \ref{fig1}  although $\pex$ (resp. $\pdrop$)    is not a mesh pattern.
See \cite{BK21, HZ21} for further information about
 mesh patterns.


Recently Baril and Kirgizov~\cite{BK21}
proved   the equidistribution of
the statistics  "$\des_2$", "$\pex$" and "$\pcyc$"  over $\Sym_n$ by
bijections and conclude their paper with  the following two conjectures  on the equidistribution of  two pairs of
bistatistics.

\begin{conjecture}[Baril and Kirgizov]\label{conj:1}
 The two bistatistics  $(\des_2,\cyc)$ and
$(\pex, \cyc)$ are equidistributed on $\Sym_n$.
\label{conj1}
\end{conjecture}

\begin{conjecture}[Vajnovszki]\label{conj:2}
 The two bistatistics  $(\des_2,\des)$ and
$(\pex, \exc)$ are equidistributed on $\Sym_n$.
\label{conj2}
\end{conjecture}

%
\begin{figure}[t]
     \begin{center}
             $\des_2=~$\scalebox{0.5}    {\begin{tikzpicture}
                 \draw [thick] (\bm,\am) -- (\bm,\dm);
                 \draw [thick] (\cm,\am) -- (\cm,\dm);
                 \draw [thick] (\am,\bm) -- (\dm,\bm);
                 \draw [thick] (\am,\cm) -- (\dm,\cm);


         \draw [fill=lightgray,pattern=north east
lines](\bm,\am)--(\bm,\dm)--(\cm,\dm)--(\cm,\am);
               \draw [fill=lightgray,pattern=north east
lines](\am,\cm)--(\am,\dm)--(\bm,\dm)--(\bm,\cm);
         \point{\bm}{\cm};\point{\cm}{\bm}
             \end{tikzpicture}}
             $\pex=~$\scalebox{0.5}    {\begin{tikzpicture}
                 \draw [thick] (\Bm,\am) -- (\Bm,\dm);
                 \draw [thick] (\am,\Bm) -- (\dm,\Bm);
                 \draw [thick] (\am,\Cm) -- (\dm,\Cm);
                 \draw [thick] (\am,\am) -- (\dm,\dm);
         \draw [fill=lightgray,pattern=north east
lines](\am,\Bm)--(\am,\Cm)--(\Bm,\Cm)--(\Bm,\Bm);
         \point{\Bm}{\Cm};
             \end{tikzpicture}}
                     $\pdrop=~$\scalebox{0.5}    {\begin{tikzpicture}
                 \draw [thick] (\Bm,\am) -- (\Bm,\dm);
                 \draw [thick] (\am,\Bm) -- (\dm,\Bm);
                 \draw [thick] (\am,\Am) -- (\dm,\Am);
                 \draw [thick] (\am,\am) -- (\dm,\dm);
         \draw [fill=lightgray,pattern=north east
lines](\Bm,\Am)--(\Bm, \Bm)--(\dm, \Bm)--(\dm, \Am);
         \point{\Bm}{\Am};
             \end{tikzpicture}}
            $\eareccpeak=~$\scalebox{0.5}    {\begin{tikzpicture}
                 \draw [thick] (\bm,\am) -- (\bm,\dm);
                 \draw [thick] (\cm,\am) -- (\cm,\dm);
                 \draw [thick] (\am,\bm) -- (\dm,\bm);
                 \draw [thick] (\am,\cm) -- (\cm,\cm);
                 \draw [thick] (\am,\am) -- (\dm,\dm);
                 \draw (\dm,\cm) node[cross] {};
                 \draw (2.2cm,\cm) node[cross] {};
                 \draw (2.4cm,\cm) node[cross] {};
                  \draw (2.6cm,\cm) node[cross] {};
                   \draw (2.8cm,\cm) node[cross] {};
         \draw [fill=lightgray,pattern=north east
lines](\cm,\am)--(\cm,\bm)--(\dm,\bm)--(\dm,\am);
         \point{\cm}{\bm};
             \end{tikzpicture}}
        \end{center}
        \medskip
     \caption{Illustration of the mesh patterns $\des_2$ and
$\pex$ and $\eareccpeak$, where the cross line means
that the value cannot  be in the segment of the horizontal line}
\label{fig1}
\end{figure}
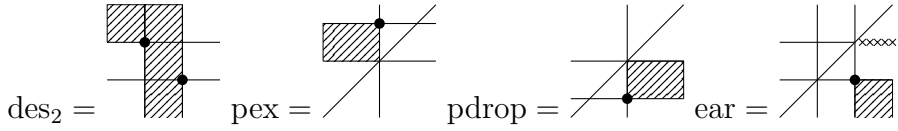

In  this paper we shall  take a different approach to their problems through   the combinatorial theory of J-continued fractions developed  by Flajolet and Viennot in the 1980's~\cite{FV79,Fl80}, see
\cite{BS20, Eli17, HMZ20, SZ20} for  recent developments of this theory.
Recall that a J-type continued fraction  is  a formal power  series defined by
\[
\sum_{n=0}^\infty a_n z^n=\frac{1}{1-\gamma_0 z-\cfrac{\beta_1 z^2}{1-\gamma_1 z-
\cfrac{\beta_2z^2}{\cdots}}},
\]
where  $(\gamma_n)_{n\geq 0}$ and
$(\beta_n)_{n\geq 1}$ are two sequences in some commutative ring.

Define 
the polynomials $A_n(t, \lambda, y, w)$  by the  J-fraction
\begin{align}\label{cf:generalA}
\sum_{n\geq 0}z^n
A_n(t, \lambda, y, w)
 \;=\;
 \cfrac{1}
 {1-wz-\cfrac{t\lambda y\,z^2}
 {1-(w+t+1)z-\cfrac{t(\lambda+1)(y+1)\,z^2}
 {\cdots}}}
 \end{align}
 with $\gamma_n=w+n(t+1)$ and $\beta_n=t(\lambda+n-1)(y+n-1)$.

It is known that $A_n(t, 1,1,1)$ equals 
the Eulerian polynomial $A_n(t)$, see~\cite{HMZ20, SZ20}.
Recently Sokal and the third author~\cite{SZ20} have generalized the
J-fraction for Eulerian polynomials in infinitly many inderterminates, which are also  generailzations of the polynomials $A_n(t, \lambda, y, w)$. The aim of this paper is 
to generalize 
the  results in \cite{BK21} by  exploring the 
combinatorial interpretations of the polynomials
$A_n(t, \lambda, y, w)$ in light of the aformentioned statistics. 
  In particular, we
confirm and strengthen Conjecture~\ref{conj:2} (see Corollary \ref{cor:4pairs}) and obtain  five
equidistributed compagnions  of  the bistatistic $(\pex, \cyc)$ in Conjecture \ref{conj:1} (see Theorem~\ref{mainresult:3}).

\subsection{Main results}

For  $\sigma\in \Sym_n$,
an index $i \in [n]$ is called (see \cite{SZ20}) a
\begin{itemize}
   \item {\em cycle peak}\/ (cpeak) if $\sigma^{-1}(i) < i > \sigma(i)$;
   \item {\em cycle valley}\/ (cval) if $\sigma^{-1}(i) > i < \sigma(i)$;
   \item {\em cycle double rise}\/ (cdrise) if $\sigma^{-1}(i) < i < \sigma(i)$;
   \item {\em cycle double fall}\/ (cdfall) if $\sigma^{-1}(i) > i > \sigma(i)$;
   \item {\em fixed point}\/ (fix) if $\sigma^{-1}(i) = i = \sigma(i)$.
\end{itemize}
Clearly every index $i$ belongs to exactly one of these five types;
we refer to this classification as the {\em cycle classification}\/.
Next,
 an  index $i \in [n]$ (or a value $\sigma(i)$) is called a
\begin{itemize}
\item {\em record}\/ (rec) (or {\em left-to-right maximum}\/)
         if $\sigma(j) < \sigma(i)$ for all $j < i$
      (The index $1$ is  always  a record];
   \item {\em antirecord}\/ (arec) (or {\em right-to-left minimum}\/)
         if $\sigma(j) > \sigma(i)$ for all $j > i$
      (The index $n$  is   always an antirecord);
   \item {\em exclusive record}\/ (erec) if it is a record and not also
         an antirecord;
   \item {\em exclusive antirecord}\/ (earec) if it is an antirecord and not also
     a record.
     \item \textit{exclusive antirecord cycle peak} ($\eareccpeakbis$) if $i$ is an exclusive antirecord and also 
     a cycle peak.
\end{itemize}
 The  statistic  $\eareccpeakbis$ was introduced in \cite{SZ20}, in this paper we adopt the  following concise notation instead
 \begin{align}
 \eareccpeak:=\eareccpeakbis.
\end{align}
An illustration of the pattern  $\ear$ is given  in Figure \ref{fig1}.
Also, we shall denote the set of  indexes of each type by
capitalizing  the first letter of type name. Hence ${\rm Cpeak}\,\sigma$ denotes  the set of
indexes of cycle peaks of $\sigma$.
For example, if  $\sigma=2\,3\,1\,4\,7\,8\,6\,5=(1\,2\,3)(4)(6\,8\, 5\, 7)$,
then
${\rm Earec}\, \sigma=\{3,  8\}$ as $\sigma(3)=1$ and $\sigma(8)=5$ and ${\rm Cpeak}\,\sigma=\{3,7,  8\}$, so
${\rm Ear}\,\sigma =\{3, 8\}$ and  $\eareccpeak\,\sigma=2$.

Our first result provides three interpretations for the polynomials
$A_n(t,\lambda, y, w)$ in \eqref{cf:generalA}.

\begin{theorem}\label{mainresult:1}
We have
\begin{subequations}
\begin{align}
A_n(t, \lambda, y, w)&=
\sum_{\sigma\in\Sym_n}t^{\exc\,\sigma}\lambda^{\pex\,\sigma}y^{\eareccpeak\,\sigma}w^{\fix\,\sigma}\label{Equ1}\\
&=\sum_{\sigma\in\Sym_n}t^{\exc\,\sigma}
\lambda^{\pcyc\,\sigma}y^{\eareccpeak\,\sigma}w^{\fix\,\sigma}\label{Equ2}\\
&=\sum_{\sigma\in\Sym_n}t^{\exc\,\sigma}\lambda^{\pcyc\,\sigma}y^{\pex\,\sigma}
w^{\fix\,\sigma}.\label{Equ3}
\end{align}
\end{subequations}
\end{theorem}

By \eqref{cf:generalA},
the polynomial  $A_n(t, \lambda, y, w)$ is invariant
under $\lambda\leftrightarrow y$. Hence, the above theorem implies immediately the following result.
\begin{corollary}
The six  bistatistics 
$(\pex, \eareccpeak)$, $(\eareccpeak, \pex)$, 
$(\ear, \pcyc)$, $(\pcyc, \ear)$,  $( \pex, \pcyc)$ and  $(\pcyc, \pex)$
are equidistributed on $\Sym_n$.
\end{corollary}

Now we consider three specializations of $A_n(t, \lambda, y, w)$.
First let
$
B_n(t,\lambda, w)=A_n(t, \lambda, 1, w)=A_n(t, 1, \lambda, w)$,
namely,
\begin{align}\label{eq:zeng93}
\sum_{n\geq 0}z^n
B_n(t,  w, \lambda)=
 \cfrac{1}
 {1-wz-\cfrac{t\lambda  \,z^2}
 {1-( w+t+1)z-\cfrac{2t(\lambda +1)\,z^2}
 {\cdots}}}
\end{align}
with $\gamma_n=w+n(t+1)$ and $\beta_n=nt(\lambda+n-1)$.

\begin{rem} By \eqref{Equ3} we recover the fix and cycle  $(p,q)$-Eulerian
polynomials~\cite{MMYY21,  Zeng93, KZ02}
\begin{align}
A_n(x, p, 1, pq)=\sum_{\sigma\in \Sym_n}
x^{\exc\, \sigma}  p^{\cyc\, \sigma}q^{\fix\, \sigma}.
\end{align}
\end{rem}

\medskip
To deal with descent statistics, we recall some linear statistics from \cite{HMZ20}.
For  $\sigma=\sigma(1)\sigma(2)\cdots\sigma(n)\in \Sym_{n}$  with convention 0--$\infty$, i.e., $\mathbf{\sigma(0)=0}$ and  $\mathbf{\sigma(n+1)=n+1}$,  a value $\sigma(i)$ ($1\leq i\leq n$) is called a
\def\asc{\textsf{asc}}
\begin{itemize}
\item  \emph{double ascent} (dasc) if $\sigma(i-1)<\sigma(i)$ and $\sigma(i)<\sigma(i+1)$;
\item  \emph{double descent} (ddes) if $\sigma(i-1)>\sigma(i)$ and $\sigma(i)>\sigma(i+1)$;
\item  \emph{peak} (peak) if $\sigma(i-1)<\sigma(i)$ and $\sigma(i)>\sigma(i+1)$;
\item  \emph{valley} (valley) if $\sigma(i-1)>\sigma(i)$ and $\sigma(i)<\sigma(i+1)$.
\end{itemize}
          A double ascent $\sigma(i)$ ($1\leq i\leq n $) is called a \emph{foremaximum} of $\sigma$ if it is at the same time a record.
        Denote the number of foremaxima of $\sigma$ by $\fmax\:\sigma$.
For example, if $\sigma=3\,4\,2\,1\,5\,8\,7\,6$, 
then $\dasc \sigma=\ddes \sigma=\peak\sigma=\valley\sigma=2$ and 
 $\fmax\,\sigma=2$ as the foremaxima of $\sigma$
are $3, 5$. 

\begin{theorem}\label{mainresult:2}
We have
\begin{subequations}
\begin{align}
B_n(t,\lambda, w)
&=\sum_{\sigma\in\Sym_n}t^{\exc\,\sigma}\lambda^{\pcyc\,\sigma}w^{\fix\,\sigma}\label{form1}\\
&=\sum_{\sigma\in\Sym_n}t^{\exc\,\sigma}\lambda^{\eareccpeak\,\sigma}w^{\fix\,\sigma}\label{form2}\\
&=\sum_{\sigma\in\Sym_n}t^{\exc\,\sigma}\lambda^{\pex\,\sigma}w^{\fix\,\sigma}\label{form3}\\
&=\sum_{\sigma\in\Sym_n}t^{\des\,\sigma}\lambda^{\des_2\,\sigma}w^{\fmax\,\sigma}\label{form4}
\end{align}
{\rm and}
\begin{equation}\label{eq:gfB}
\sum_{n\geq0}B_n(t, w, \lambda)\frac{z^n}{n!}=e^{w z}\left(\frac{1-t}{e^{tz}-te^z}\right)^{\lambda }.
\end{equation}
\end{subequations}
\end{theorem}

The following corollary of Theorem~\ref{mainresult:2} confirms and generalizes  Conjecture~\ref{conj:2}. 
\begin{corollary}\label{cor:4pairs} The  four  bistatistics
$(\exc,\,\pcyc), \, (\exc,\, \ear),\, (\des,\,\des_2)$ and  $(\exc,\,\pex)$ are equidistributed over $\Sym_n$.
\end{corollary}
\begin{rem}
We will provide 
bijective proofs of   
Corollary~\ref{cor:4pairs} in Lemma \ref{des2=pcyc} and 
Theorem \ref{thm1.8}.
\end{rem}

Next let
$
C_n(y, \lambda)=A_n(1, \lambda, y, \lambda)=A_n(1,  y, \lambda,\lambda)$. We obtain the following result directly from Theorem~\ref{mainresult:2}. 
\begin{theorem}\label{mainresult:3}
We have
\begin{subequations}
\begin{align}
C_n(y, \lambda)&=\sum_{\sigma\in\Sym_n}
y^{\pex\,\sigma}\lambda^{\eareccpeak\,\sigma+\fix\,\sigma}
=\sum_{\sigma\in\Sym_n}
y^{\eareccpeak\,\sigma}\lambda^{\pex\,\sigma+\fix\,\sigma}\label{eq:c2}\\
&=\sum_{\sigma\in\Sym_n}
y^{\pcyc\,\sigma}\lambda^{\eareccpeak\,\sigma+\fix\,\sigma}
=\sum_{\sigma\in\Sym_n}
y^{\eareccpeak\,\sigma}\lambda^{\cyc\,\sigma}\label{eq:c4}\\
&=\sum_{\sigma\in\Sym_n}
y^{\pcyc\,\sigma}\lambda^{\pex\,\sigma+\fix\,\sigma}
=\sum_{\sigma\in\Sym_n} y^{\pex\,\sigma}\lambda^{\cyc\,\sigma}.\label{eq:c6}
\end{align}
\end{subequations}
\end{theorem}

Finally let $D_n(t, \lambda, y)=A_n(t, \lambda, y, 0)$. From Theorem~\ref{mainresult:1} we deduce
\begin{subequations}
\begin{align}
D_n(t, \lambda, y)&=
\sum_{\sigma\in\SDE_n}t^{\exc\,\sigma}\lambda^{\pex\,\sigma}y^{\eareccpeak\,\sigma}\label{dEqu1}\\
&=\sum_{\sigma\in\SDE_n}t^{\exc\,\sigma}
\lambda^{\cyc\,\sigma}y^{\eareccpeak\,\sigma}\label{dEqu2}\\
&=\sum_{\sigma\in\SDE_n}t^{\exc\,\sigma}\lambda^{\cyc\,\sigma}y^{\pex\,\sigma},
\label{dEqu3}
\end{align}
\end{subequations}
where $\SDE_n$  is  the set of 
derangements in $\Sym_n$.

Let $\SDE_n^*$ the subset 
of  $\SDE_n$ consisting of derangements without cycle double rise. Furthermore, for $k\in [n]$ define the set
\begin{align}\label{dstar}
\SDE_n^*(k)=&\{\sigma\in \SDE_n\mid \exc(\sigma)=k,\; \cdrise(\sigma)=0\}.
\end{align}
 We show that  the polynomials $D_n(t, \lambda, y)$
 have a nice $\gamma$-positive formula, see \cite{Bra08, HMZ20} for further informations.
\begin{theorem}\label{gamma-thm}
 We have 
\begin{equation}\label{gamma-D}
D_n(t, \lambda, y)=\sum_{k=0}^{\left\lfloor \frac{n}{2}\right\rfloor}\gamma_{n,k}(\lambda, y)
t^k(1+t)^{n-2k},
\end{equation}
where the gamma coefficient $\gamma_{n,k}(\lambda, y)$ has the following interpreattions
\begin{subequations}
\begin{align}
\gamma_{n,k}(\lambda, y)&=\sum_{\sigma\in\SDE_n^*(k)}\lambda^{\pex\,\sigma}y^{\eareccpeak\,\sigma}\label{gamma1}\\
&=\sum_{\sigma\in\SDE_n^*(k)}\lambda^{\cyc\,\sigma}y^{\eareccpeak\,\sigma}\label{gamma2}\\
&=\sum_{\sigma\in\SDE_n^*(k)}\lambda^{\cyc\,\sigma}y^{\pex\,\sigma}.\label{gamma3}
\end{align}
\end{subequations}
\end{theorem}
\begin{rem} For $\sigma\in \SDE_n^*(k)$, the mapping  $\sigma\mapsto \sigma^{-1}$ is a bijection from $\SDE_n^*(k)$ to
$\SDE_n^{**}(k)$ with 
\begin{align}
\SDE_n^{**}(k)=&\{\sigma\in \SDE_n\mid \drop(\sigma)=k,\; \cdfall(\sigma)=0\}.
\end{align}
Thus, when $y=1$ both \eqref{gamma2} and \eqref{gamma3} reduce to \cite[Theorem 11]{SZ12}.
\end{rem}
The rest of this paper is organized as follows: we first prove 
the main results in Section~2.
In Section~3, 
we will construct two  bijections on $\Sym_n$
 to prove the equality between \eqref{form3} and \eqref{form4}, namely we have the following  result.

\begin{theorem}\label{thm1.8} There are bijections $\Phi_1: \Sym_n\to \Sym_n$ and  $\Phi_2: \Sym_n\to \Sym_n$ such that
\begin{subequations}
\begin{align}
(\des, \des_2)\,\sigma
&=(\exc, \ear)\,\Phi_1 (\sigma);\label{des2-ear}\\
(\des, \des_2, \fmax)\,\sigma
&=(\exc, \pex, \fix)\,\Phi_2 (\sigma).\label{des2-pex}
\end{align}
\end{subequations}
\end{theorem}

	\begin{rem}
	Note that $\Phi_2$ gives a bijective proof of Conjecture~\ref{conj:2}.
	\end{rem}

The proof of Theorem~\ref{mainresult:1} 
relies on  the fact the polynomial  $A_n(t,\lambda, y, w)$ can be obtained by specializing  the two master 
polynomials $Q_n$ and 
$\widehat{Q}_n$ in \cite{SZ20}, see \eqref{def.Qn.firstmaster} and  \eqref{def.Qn.secondmaster}.
Actually, we shall derive  Theorem~\ref{mainresult:1} 
from the \emph{first and second master J-fractions for
permutations}  in \cite[Theorems 2.9 and 2.14]{SZ20},
  and a dual form of the second master J-fraction, see Proposition~\ref{dualmaster2}. For reader's convenience we shall recall the two master J-fractions for permutations in the next section.
  
As  the polynomials $Q_n$ and 
$\widehat{Q}_n$  are originally defined  using cyclic statistics of permutations,   it is then  suggested in \cite{SZ20} to  seek for  interpretations using linear statistics  for these master  polynomials.  
In Section~3,  we shall give 
two such  interpretations  for the  polynomials $Q_n$  and as an application, we  give a group action proof 
for  a gamma-expansion formula Eq.\eqref{gamma1} (see  Theorem~\ref{gamma-thm}). Finally we conclude the paper with some open questions.

 \subsection{Two  master J-fractions for permutations}
 We recall the two master J-fractions for permutations in \cite{SZ20}.
 First we associate to each permutation $\sigma\in \Sym_n$ a pictorial representation (Figure \ref{fig.pictorial}) by placing vertices $1, 2, \ldots, n$ along horizontal axis and then draw an arc from $i$ to $\sigma(i)$ above (resp. below) the horizontal axis in case $\sigma(i)>i$ (resp. $\sigma(i)<i$), if $\sigma(i)=i$ we do not draw any arc. Of course, the arrows on the arc are redundunt, because the arrow on an arc above (resp. below) the axis always point to the right (resp. left).
We then say that a quadrauplet $i<j<k<l$ forms an
\begin{itemize}
\item \emph{upper crossing} ($\ucross$) if $k=\sigma(i)$ and $l=\sigma(j)$;
\item \emph{lower crossing} ($\lcross$) if $i=\sigma(k)$ and $j=\sigma(l)$;
\item \emph{upper nesting} ($\unest$) if $l=\sigma(i)$ and $k=\sigma(j)$;
\item \emph{lower nesting} ($\lnest$) if $i=\sigma(l)$ and $j=\sigma(k)$.
\end{itemize}
See Figure \ref{fig: crossing} and Figure \ref{fig: nesting}.
We also need a refined version of the above statistics.
The basic idea is that,
rather than counting the {\em total}\/ numbers of quadruplets
$i < j < k < l$ that form upper (resp.\ lower) crossings or nestings,
we should instead count the number of upper (resp.\ lower) crossings or nestings
that use a particular vertex $j$ (resp.\ $k$)
in second (resp.\ third) position,
and then attribute weights to the vertex $j$ (resp.\ $k$)
depending on those values.
More precisely, we define
\begin{subeqnarray}
   \ucross(j,\sigma)
   & = &
   \#\{ i<j<k<l \colon\: k = \sigma(i) \hbox{ and } l = \sigma(j) \}
         \\[2mm]
   \unest(j,\sigma)
   & = &
   \#\{ i<j<k<l \colon\: k = \sigma(j) \hbox{ and } l = \sigma(i) \}
      \slabel{def.unestjk}     \\[2mm]
   \lcross(k,\sigma)
   & = &
   \#\{ i<j<k<l \colon\: i = \sigma(k) \hbox{ and } j = \sigma(l) \}
         \\[2mm]
   \lnest(k,\sigma)
   & = &
   \#\{ i<j<k<l \colon\: i = \sigma(l) \hbox{ and } j = \sigma(k) \}
 \label{def.ucrossnestjk}
\end{subeqnarray}
 We also consider the degenerate cases with $j=k$, by saying that a triplet $i<j<l$ forms an 
\begin{itemize}
\item \emph{upper pseudo-nesting} (upsnest) if
$l=\sigma(i)$ and $j=\sigma(j)$;
\item \emph{lower pseudo-nesting} (lpsnest) if
$i=\sigma(l)$ and $j=\sigma(j)$.
\end{itemize}
See Figure \ref{fig: pseudo-nesting}.
Note that ${\rm upsnest}(\sigma)={\rm lpsnest}(\sigma)$ for all $\sigma$ (see \cite{SZ20}). We therefore write these two statistics simply as
$$
\lev(\sigma)={\rm upsnest}(\sigma)={\rm lpsnest}(\sigma).
$$
The refined level of a fixed point $j$ ($\sigma(j)=j$) is defined by
\be
   \lev(j,\sigma)
   \;=\;
   \#\{ i<j<l \colon\: l = \sigma(i) \}
   \;=\;
   \#\{ i<j<l \colon\: i = \sigma(l) \}
   \;.
 \label{def.level.bis}
\ee
And we obviously have
\be
   \ucross(\sigma)
   \;=\;
   \sum\limits_{j \in {\rm cval}}  \ucross(j,\sigma)
 \label{eq.ucrosscval.sum}
\ee
and analogously for the other four statistics 
$\lcross$, $\unest$, $\lnest$ and $\lev$.
 \begin{figure}[t]
 \begin{picture}(120,60)(140, -40)
\setlength{\unitlength}{2mm}
\linethickness{.5mm}
\put(2,0){\line(1,0){28}}
\put(5,0){\circle*{1,3}}\put(5,0){\makebox(0,-6)[c]
{\small i}}
\put(12,0){\circle*{1,3}}\put(12,0){\makebox(0,-6)[c]
{\small j}}
\put(19,0){\circle*{1,3}}\put(19,0){\makebox(0,-6)[c]
{\small $k$}}
\put(26,0){\circle*{1,3}}\put(26,0){\makebox(0,-6)[c]
{\small $l$}}
\red{\qbezier(5,0)(12,10)(19,0)}
\blue{\qbezier(11,0)(18,10)(25,0)}
\put(33,0){\line(1,0){28}}
\put(37,0){\circle*{1,3}}\put(37,0){\makebox(0,-6)[c]{\small $i$}}\put(44,0){\circle*{1,3}}\put(44,0){\makebox(0,-6)[c]{\small $j$}}\put(51,0){\circle*{1,3}}\put(51,0){\makebox(0,-6)[c]{\small $k$}}
\put(58,0){\circle*{1,3}}\put(58,0){\makebox(0,-6)[c]
{\small $l$}}
\red{\qbezier(37,0)(44,-10)(51,0)}
\blue{\qbezier(43,0)(50,-10)(57,0)}
\end{picture}
\caption{\label{fig: crossing} Upper crossing and lower crossing}
\end{figure}
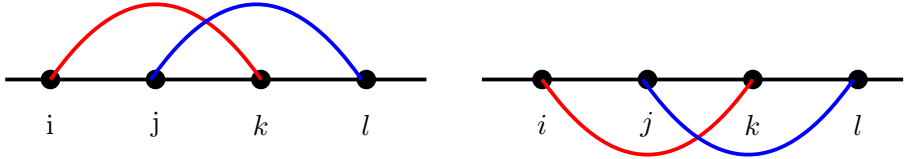

 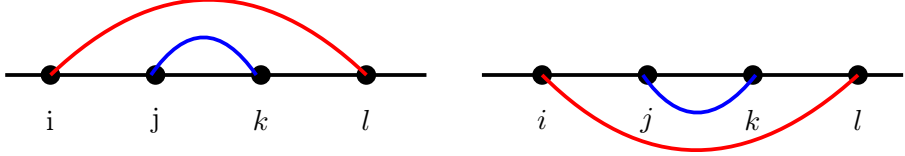
\begin{figure}[t]
 \begin{picture}(120,60)(140, -40)
\setlength{\unitlength}{2mm}
\linethickness{.5mm}
\put(2,0){\line(1,0){28}}
\put(5,0){\circle*{1,3}}\put(5,0){\makebox(0,-6)[c]
{\small i}}
\put(12,0){\circle*{1,3}}\put(12,0){\makebox(0,-6)[c]
{\small j}}
\put(19,0){\circle*{1,3}}\put(19,0){\makebox(0,-6)[c]
{\small $k$}}
\put(26,0){\circle*{1,3}}\put(26,0){\makebox(0,-6)[c]
{\small $l$}}
\red{\qbezier(5,0)(15.5,10)(26,0)}
\blue{\qbezier(11,0)(14.5,5)(18,0)}
\put(33,0){\line(1,0){28}}
\put(37,0){\circle*{1,3}}\put(37,0){\makebox(0,-6)[c]{\small $i$}}\put(44,0){\circle*{1,3}}\put(44,0){\makebox(0,-6)[c]{\small $j$}}\put(51,0){\circle*{1,3}}\put(51,0){\makebox(0,-6)[c]{\small $k$}}
\put(58,0){\circle*{1,3}}\put(58,0){\makebox(0,-6)[c]
{\small $l$}}
\red{\qbezier(37,0)(47,-10)(58,0)}
\blue{\qbezier(43,0)(46.5,-5)(50.5,0)}
\end{picture}
\caption{\label{fig: nesting} Upper nesting and lower nesting}
\end{figure}
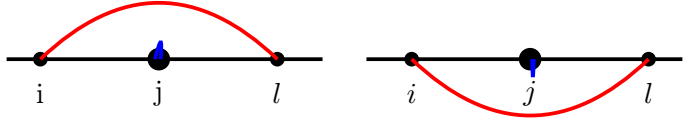
\begin{figure}[t]
 	\centering
 \begin{picture}(120,60)(100, -40)
\setlength{\unitlength}{1.5mm}
\linethickness{.5mm}
\put(2,0){\line(1,0){28}}
\put(5,0){\circle*{1,3}}\put(5,0){\makebox(0,-6)[c]{\small i}}
\put(15.5,0){\circle*{2}}\put(15.5,0){\makebox(0,-6)[c]
{\small j}}
\put(26,0){\circle*{1,3}}\put(26,0){\makebox(0,-6)[c]
{\small $l$}}
\red{\qbezier(5,0)(15.5,10)(26,0)}
\blue{\qbezier(14.25,0)(15,3)(14.75,0)}
\put(33,0){\line(1,0){28}}
\put(37,0){\circle*{1.3}}
\put(37,0){\makebox(0,-6)[c]{\small $i$}}
\put(47.5,0){\circle*{2}}
\put(47.5,0){\makebox(0,-6)[c]{\small $j$}}
\put(58,0){\circle*{1,3}}\put(58,0){\makebox(0,-6)[c]
{\small $l$}}
\red{\qbezier(37,0)(47.5,-10)(58,0)}
\blue{\qbezier(46.75,0)(46.8,-3)(46.9,0)}
\end{picture}
\caption{\label{fig: pseudo-nesting} Upper pseudo-nesting and lower pseudo-nesting of a fixed point}
\end{figure}


We now introduce five infinite families of indeterminates
$\bsfa = (\sfa_{\ell,\ell'})_{\ell,\ell' \ge 0}$,
$\bsfb = (\sfb_{\ell,\ell'})_{\ell,\ell' \ge 0}$,
$\bsfc = (\sfc_{\ell,\ell'})_{\ell,\ell' \ge 0}$,
$\bsfd = (\sfd_{\ell,\ell'})_{\ell,\ell' \ge 0}$,
$\bsfe = (\sfe_\ell)_{\ell \ge 0}$
and define the polynomial $Q_n(\bsfa,\bsfb,\bsfc,\bsfd,\bsfe)$ by
\begin{eqnarray}
   & & \hspace*{-10mm}
   Q_n(\bsfa,\bsfb,\bsfc,\bsfd,\bsfe)
   \;=\;
       \nonumber \\[4mm]
   & &
   \sum_{\sigma \in \Sym_n}
   \;\:
   \prod\limits_{i \in {\rm Cval}}  \! \sfa_{\ucross(i,\sigma),\,\unest(i,\sigma)}
   \prod\limits_{i \in {\rm Cpeak}} \!\!  \sfb_{\lcross(i,\sigma),\,\lnest(i,\sigma)}
       \:\times
       \qquad\qquad
       \nonumber \\[1mm]
   & & \qquad\;
   \prod\limits_{i \in {\rm Cdfall}} \!\!  \sfc_{\lcross(i,\sigma),\,\lnest(i,\sigma)}
   \;
   \prod\limits_{i \in {\rm Cdrise}} \!\!  \sfd_{\ucross(i,\sigma),\,\unest(i,\sigma)}
   \, \prod\limits_{i \in {\rm Fix}} \sfe_{\lev(i,\sigma)}
   \;.
   \quad
 \label{def.Qn.firstmaster}
\end{eqnarray}
The following is the first master J-fraction for permutations in \cite[Theorem~2.9]{SZ20}.

\begin{theorem}\cite{SZ20}
   \label{thm.permutations.Jtype.final1}
The ordinary generating function of the polynomials
$Q_n(\bsfa,\bsfb,\bsfc,\bsfd,\bsfe)$
has the J-type continued fraction
\begin{eqnarray}
   & & \hspace*{-8mm}
   \sum_{n=0}^\infty Q_n(\bsfa,\bsfb,\bsfc,\bsfd,\bsfe) \: z^n
   \;=\;
       \nonumber \\
   & & \hspace*{-4mm}
\Scale[0.8]{
   \cfrac{1}{1 - \sfe_0 z - \cfrac{\sfa_{00} \sfb_{00} z^2}{1 -  (\sfc_{00} + \sfd_{00} + \sfe_1) z - \cfrac{(\sfa_{01} + \sfa_{10})(\sfb_{01} + \sfb_{10}) z^2}{1 - (\sfc_{01} + \sfc_{10} + \sfd_{01} + \sfd_{10} + \sfe_2)z - \cfrac{(\sfa_{02} + \sfa_{11} + \sfa_{20})(\sfb_{02} + \sfb_{11} + \sfb_{20}) z^2}{1 - \cdots}}}}
}
       \nonumber \\[1mm]
   \label{eq.thm.permutations.Jtype.final1}
\end{eqnarray}
with coefficients
\begin{subeqnarray}
   \gamma_n  & = &   \sfc^\star_{n-1} \,+\, \sfd^\star_{n-1} \,+\, \sfe_n
          \\[1mm]
   \beta_n   & = &   \sfa^\star_{n-1} \, \sfb^\star_{n-1}
 \label{def.weights.permutations.Jtype.final1}
\end{subeqnarray}
where
\be
   \sfa^\star_{n-1}  \;\eqdef\;  \sum_{\ell=0}^{n-1} \sfa_{\ell,n-1-\ell}
 \label{def.astar}
\ee
and likewise for $\sfb,\sfc,\sfd$.
\end{theorem}

We again introduce five infinite families of indeterminates:
$\bsfa = (\sfa_{\ell})_{\ell \ge 0}$,
$\bsfb = (\sfb_{\ell,\ell'})_{\ell,\ell' \ge 0}$,
$\bsfc = (\sfc_{\ell,\ell'})_{\ell,\ell' \ge 0}$,
$\bsfd = (\sfd_{\ell,\ell'})_{\ell,\ell' \ge 0}$,
$\bsfe = (\sfe_\ell)_{\ell \ge 0}$;
please note that $\bsfa$ now has one index rather than two.
We then define the polynomial
$\widehat{Q}_n(\bsfa,\bsfb,\bsfc,\bsfd,\bsfe,\lambda)$ by
\begin{eqnarray}
   & & \hspace*{-10mm}
   \widehat{Q}_n(\bsfa,\bsfb,\bsfc,\bsfd,\bsfe,\lambda)
   \;=\;
       \nonumber \\[4mm]
   & &
   \sum_{\sigma \in \Sym_n}
   \;\:
   \lambda^{\cyc(\sigma)} \;
   \prod\limits_{i \in {\rm Cval}}  \! \sfa_{\ucross(i,\sigma)+\unest(i,\sigma)}
   \prod\limits_{i \in {\rm Cpeak}} \!\!  \sfb_{\lcross(i,\sigma),\,\lnest(i,\sigma)}
       \:\times
       \qquad\qquad
       \nonumber \\[1mm]
   & & \;
   \prod\limits_{i \in {\rm Cdfall}} \!\!  \sfc_{\lcross(i,\sigma),\,\lnest(i,\sigma)}
   \,
   \prod\limits_{i \in {\rm Cdrise}} \!\!  \sfd_{\ucross(i,\sigma)+\unest(i,\sigma),\,\unest(\sigma^{-1}(i),\sigma)}
   \, \prod\limits_{i \in {\rm Fix}} \sfe_{\lev(i,\sigma)}
   \;.
   \qquad
 \label{def.Qn.secondmaster}
\end{eqnarray}
Note that here, in contrast to Theorem~\ref{thm.permutations.Jtype.final1},
$\widehat{Q}_n$ depends on $\ucross(i,\sigma)$ and $\unest(i,\sigma)$
only via their sum;
and note also the somewhat bizarre appearance of
$\unest(\sigma^{-1}(i),\sigma)$ as the second index on $\sfd$.
The following is the second
 master J-fraction for permutations in \cite[Theorem~2.14]{SZ20}.

\begin{theorem}\cite{SZ20}
   \label{thm.permutations.Jtype.final2} 
The ordinary generating function of the polynomials
$\widehat{Q}_n(\bsfa,\bsfb,\bsfc,\bsfd,\bsfe,\lambda)$
has the J-type continued fraction
\begin{eqnarray}
   & & \hspace*{-8mm}
   \sum_{n=0}^\infty \widehat{Q}_n(\bsfa,\bsfb,\bsfc,\bsfd,\bsfe,\lambda) \: z^n
   \;=\;
       \nonumber \\
   & & \hspace*{-4mm}
\Scale[0.85]{
   \cfrac{1}{1 - \lambda\sfe_0 z - \cfrac{\lambda \sfa_0 \sfb_{00} z^2}{1 -  (\sfc_{00} + \sfd_{00} + \lambda\sfe_1) z - \cfrac{(\lambda+1) \sfa_1 (\sfb_{01} + \sfb_{10}) z^2}{1 - (\sfc_{01} + \sfc_{10} + \sfd_{10} + \sfd_{11} + \lambda\sfe_2)z - \cfrac{(\lambda+2) \sfa_2 (\sfb_{02} + \sfb_{11} + \sfb_{20}) z^2}{1 - \cdots}}}}
}
       \nonumber \\[1mm]
   \label{eq.thm.permutations.Jtype.final2}
\end{eqnarray}
with coefficients
\begin{subeqnarray}
   \gamma_n  & = &    \sum_{\ell=0}^{n-1} \sfc_{\ell,n-1-\ell} \,+\, \sum_{\ell=0}^{n-1} \sfd_{n-1,\ell}
                                      \,+\, \lambda\sfe_n
          \\[1mm]
   \beta_n   & = &   (\lambda+n-1) \, \sfa_{n-1} \, \sum_{\ell=0}^{n-1} \sfb_{\ell,n-1-\ell}
 \label{def.weights.permutations.Jtype.final2}
\end{subeqnarray}
\end{theorem}


%

For $\sigma                                                                                                                                                                                           \in\Sym_{n}$,
we define the  \emph{complementation} $\sigma^c$ and \emph{reversal} $\sigma^r$ of of $\sigma$ by
\begin{align}\label{c-r}
\sigma^c(i)=n+1-\sigma(i)\quad \textrm{and}\quad \sigma^r(i)=\sigma(n+1-i)\quad\textrm{ for}\quad i\in [n].
\end{align}
Let $\zeta: \sigma\mapsto \tau$ be  the transformation by reversal combined with  complementation, i.e.,
\begin{align}\label{zeta}
\zeta(\sigma)=(\sigma^r)^c.
\end{align}
By the pictorial representation of $\sigma$
we can visualize the operation $\zeta$ by
a geometric interpretation, i.e., rotate a graphical representation of a permutation by $180$ degrees,
see Figure \ref{fig.pictorial}.

\begin{figure}[t]
	\centering
	\vspace*{2cm}
\begin{picture}(160,15)(75, -50)
	\setlength{\unitlength}{2mm}
	\linethickness{.5mm}
	\put(-2,0){\line(1,0){34}}
	\put(0,0){\circle*{1,3}}\put(0,0){\makebox(0,-4)[c]{\small 1}}
	\put(4,0){\circle*{1,3}}\put(4,0){\makebox(0,-4)[c]{\small 2}}
	\put(8,0){\circle*{1,3}}\put(8,0){\makebox(0,-4)[c]{\small 3}}
	\put(12,0){\circle*{1,3}}\put(12,0){\makebox(0,-4)[c]{\small 4}}
	\put(16,0){\circle*{1,3}}\put(16,0){\makebox(0,-4)[c]{\small 5}}
	\put(20,0){\circle*{1,3}}\put(20,0){\makebox(0,-4)[c]{\small 6}}
	\put(24,0){\circle*{1,3}}\put(24,0){\makebox(0,-4)[c]{\small 7}}
	\put(28,0){\circle*{1,3}}\put(28,0){\makebox(0,-4)[c]{\small 8}}
	
	\blue{\qbezier(0,0)(8,10)(16,0)
	\qbezier(4,0)(14,10)(24,0)
	\qbezier(16,0)(22,10)(28,0)}
	\red{
     \qbezier(7,0)(4,-6)(-1,0)
	\qbezier(20,0)(12,-10)(4,0)
	\qbezier(23.4,0)(22,-6)(19,0)
	\qbezier(27.4,0)(18,-10)(7,0)}
	
	\end{picture}
	\put(-35,50){\makebox(5,5){$\qquad\xlongrightarrow{\zeta}\qquad$}}
	\begin{picture}(100,15)(160, -50)
	\setlength{\unitlength}{2mm}
	\linethickness{.5mm}
	\put(26,0){\line(1,0){34}}
	\put(30,0){\circle*{1,3}}\put(30,0){\makebox(0,-4)[c]{\small 1}}
	\put(34,0){\circle*{1,3}}\put(34,0){\makebox(0,-4)[c]{\small 2}}
	\put(38,0){\circle*{1,3}}\put(38,0){\makebox(0,-4)[c]{\small 3}}
	\put(42,0){\circle*{1,3}}\put(42,0){\makebox(0,-4)[c]{\small 4}}
	\put(46,0){\circle*{1,3}}\put(46,0){\makebox(0,-4)[c]{\small 5}}
	\put(50,0){\circle*{1,3}}\put(50,0){\makebox(0,-4)[c]{\small 6}}
	\put(54,0){\circle*{1,3}}\put(54,0){\makebox(0,-4)[c]{\small 7}}
	\put(58,0){\circle*{1,3}}\put(58,0){\makebox(0,-4)[c]{\small 8}}
	
	\red{\qbezier(30,0)(40,10)(50,0)
	\qbezier(34,0)(36,6)(38,0)
	\qbezier(38,0)(46,10)(54,0)
	\qbezier(50,0)(54,6)(58,0)}
	\blue{\qbezier(41.4,0)(36,-10)(30,0)
	\qbezier(54,0)(44,-10)(34,0)
	\qbezier(58,0)(50,-10)(41,0)}
	\end{picture}	
	
	\caption{Pictorial representation of $\sigma=5\, 7 \, 1 \, 4 \, 8 \, 2 \, 6 \,3$ (left)
	 and $\zeta(\sigma)=6\,3\,7\,1\,5\,8\,2\,4$.	
		\label{fig.pictorial}
	}
\end{figure}
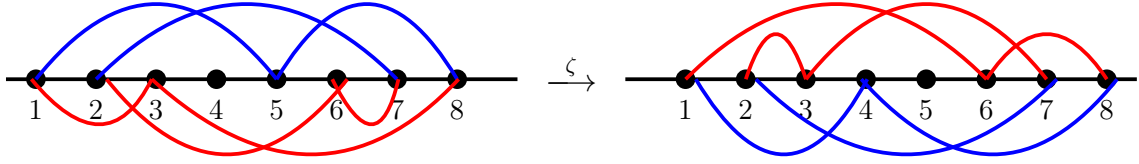




	\begin{lemma}\label{lem:exc-drop1}
	 For $\sigma\in\Sym_n$,
		we have
		\begin{align}\label{mapping: zeta}
	(\drop, \pdrop, \fix)	\sigma=(\exc, \pex, \fix) \zeta(\sigma).
		\end{align}
	\end{lemma}
\begin{proof}
This is obvious by the geometric interpretation of $\zeta$.
\end{proof}
	
We derive a dual version of Theorem~\ref{thm.permutations.Jtype.final2} from
\eqref{def.Qn.secondmaster}.
\begin{proposition} [Dual form of Theorem~\ref{thm.permutations.Jtype.final2}]\label{dualmaster2} We have
\begin{eqnarray}
   & & \hspace*{-10mm}
   \widehat{Q}_n(\bsfa,\bsfb,\bsfc,\bsfd,\bsfe,\lambda)
   \;=\;
       \nonumber \\[4mm]
   & &
   \sum_{\sigma \in \Sym_n}
   \;\:
   \lambda^{\cyc(\sigma)} \;
   \prod\limits_{i \in {\rm Cval}}  \! \sfb_{\ucross(i,\sigma),\,{\rm unest}(i,\sigma)}
   \prod\limits_{i \in {\rm Cpeak}} \!\!  \sfa_{\lcross(i,\sigma)+\lnest(i,\sigma)}
       \:\times
       \qquad\qquad
       \nonumber \\[1mm]
   & & \;
   \prod\limits_{i \in {\rm Cdfall}} \!\!  \sfd_{\lcross(i,\sigma)+\lnest(i,\sigma), \,\lnest(\sigma^{-1}(i),\sigma) }
   \,
   \prod\limits_{i \in {\rm Cdrise}} \!\!  \sfc_{\ucross(i,\sigma),\,{\rm unest}(i,\sigma)}
   \, \prod\limits_{i \in {\rm Fix}} \sfe_{{\rm lev}(i,\sigma)}
   \;.
   \qquad
 \label{def.Qn.secondmasterdual}
\end{eqnarray}
\end{proposition}

\begin{proof}
For $\sigma\in \Sym_n$
let $\tau=\zeta(\sigma)$ be the reversal combined with complementation of $\sigma$. For  $i\in [n]$  let  $i^c=n+1-i$, then
the following properties are obvious (Figure \ref{fig.pictorial})
\begin{itemize}
\item $i\in {\rm Cval}(\tau)\Leftrightarrow i^c\in {\rm Cpeak}(\sigma)$;
\item $i\in {\rm Cdrise}(\tau)\Leftrightarrow i^c\in {\rm Cdfall}(\sigma)$;
\item $i\in {\rm Fix}(\tau)\Leftrightarrow i^c\in {\rm Fix}(\sigma)$;
\end{itemize}
and
\begin{itemize}
	\item $\ucross(i,\tau)=\lcross(i^c,\sigma)$ and  $\lcross(i,\tau)=\ucross(i^c,\sigma)$;
	\item ${\rm unest}(i,\tau)=\lnest(i^c,\sigma)$ and
	$\lnest(i,\tau)={\rm unest}(i^c,\sigma)$;
	\item $\lev(i,\tau)=\lev(i^c,\sigma)$;
	\item ${\rm unest}(\tau^{-1}(i),\tau)=\lnest(\sigma^{-1}(i^c),\sigma)$ and
 $\lnest(\tau^{-1}(i),\tau)={\rm unest}(\sigma^{-1}(i^c),\sigma)$.
\end{itemize}
Moreover,  the geometric interpretation of $\zeta$ 
implies straightforwardly 
that $\cyc(\tau)=\cyc(\sigma)$.
Thus  we derive Eq.\eqref {def.Qn.secondmasterdual} 
from Eq. \eqref{def.Qn.secondmaster}.
\end{proof}

\section{Proof of the main results}

\subsection{Proof of Theorem~\ref{mainresult:1}}


\begin{lemma}\label{lem:pexear}For $\sigma\in \Sym_n$,
we have
\begin{enumerate}[a)]
\item $\exc\,\sigma=\cval\,\sigma+{\rm cdrise}\,\sigma$;
\item
$i\in [n]$ is a pure excedance of $\sigma$ if and only if
 $i$ is a cval and $\ucross(i, \sigma)=0$;
 \item
$i\in [n]$ is an eareccpeak (ear) of $\sigma$ if and only if
 $i$ is a cpeak and $\lnest(i, \sigma)=0$.
\end{enumerate}
\end{lemma}
\begin{proof} Let $\sigma\in \Sym_n$ and $i\in [n]$.
\begin{enumerate}[a)]
\item If $i$ is an excedance, i.e., $\sigma(i)>i$, then 
$\sigma^{-1}(i)>i$ or $\sigma^{-1}(i)<i$, namely
 $i$ is either a double rise or a cycle valley. 
 Inversely, if  $i$ is a double rise or a cycle valley, then $\sigma(i)>i$. 
\item
By definition, $i\in [n]$ is a pure excedance of $\sigma$ if and only if
$\sigma(i)>i$, $\sigma^{-1}(i)>i$ and $\forall j<i$ we have 
$\sigma(j)\notin [i, \sigma(i)]$, that means 
 $i$ is a cval and $\ucross(i, \sigma)=0$.
 \item By definition, 
$i\in [n]$ is an eareccpeak of $\sigma$ if and only if
there are  $j<i$ and $k<i$ such that 
$\sigma(j)=i$ and $\sigma(i)=k$
 and $\forall l>i,\; \sigma(l)>\sigma(i)$, that is, 
 $i$ is an exlusive anitirecord and  cycle peak with 
 $\lnest(i, \sigma)=0$.
\end{enumerate}
\end{proof}

Now we are ready to prove Theorem~\ref{mainresult:1}.
Our idea is to first  specialize the J-fractions in Theorems \ref{thm.permutations.Jtype.final1} 
and \ref {thm.permutations.Jtype.final2} to obtain Eq.~\eqref{cf:generalA} and then apply the corresponding combinatorial interpretations.
\begin{itemize}
\item[a)] For Eq. \eqref{Equ1}
we obtain  the J-fraction~\eqref{cf:generalA} by
taking the following substitutions in \eqref{eq.thm.permutations.Jtype.final1},
\begin{equation}\label{case1}
\begin{cases}
\sfa_{\ell, \ell'}&=t\,(\ell> 0), \quad \sfa_{0, \ell'}=\lambda t;\\
\sfb_{\ell, \ell'}&=1\,(\ell'> 0),\quad \sfb_{\ell, 0}=y;\\
\sfc_{\ell, \ell'}&=1, \quad\sfd_{\ell, \ell'}=t, \quad\sfe_{\ell}=w.
\end{cases}
\end{equation}
Then it is easy to see that Eq.~\eqref{def.Qn.firstmaster} reduces to the combinatorial interpretation  \eqref{Equ1} by Lemma~\ref{lem:pexear}.
\item[b)] For \eqref{Equ2}
we obtain  the J-fraction~\eqref{cf:generalA} by
taking the following substitutions in \eqref{eq.thm.permutations.Jtype.final2},
\begin{equation}\label{case2}
\begin{cases}
\sfa_{\ell}&=t, \quad\sfb_{\ell, \ell'}=1\,(\ell'> 0), \quad\sfb_{\ell, 0}=y;\\
\sfc_{\ell, \ell'}&=1, \quad\sfd_{\ell, \ell'}=t, \quad\sfe_{\ell}=w/\lambda, 
\end{cases}
\end{equation}
Then  Eq.~\eqref{def.Qn.secondmaster} reduces to the combinatorial interpretation \eqref{Equ2} 
by Lemma~\ref{lem:pexear}.
\item[c)] For  \eqref{Equ3}
we first obtain the J-fraction~\eqref{cf:generalA} by
taking the following substitutions in \eqref{eq.thm.permutations.Jtype.final2},
\begin{equation}
\begin{cases}\label{case3}
\sfa_{\ell}&=1, \quad \sfb_{\ell, \ell'}=t\,(\ell> 0),\quad  \sfb_{0, \ell'}=ty;\\
\sfc_{\ell, \ell'}&=t, \quad \sfd_{\ell, \ell'}=1,
\quad\sfe_{\ell}=w/\lambda,
\end{cases}
\end{equation}
Then  Eq.~\eqref{def.Qn.secondmasterdual} reduces to the combinatorial interpretation  \eqref{Equ3} by Lemma~\ref{lem:pexear} and Eq.~\eqref{eq:defpcyc}.
\end{itemize}
\qed

\subsection{Proof of Theorem~\ref{mainresult:2}}
For  $\sigma=\sigma(1)\sigma(2)\cdots\sigma(n)\in \Sym_{n}$  with convention $\infty$--0, i.e., $\sigma(0)=\infty$ and  $\sigma(n+1)=0$,
an index  $i<n$ is an ascent
if $\sigma(i)<\sigma(i+1)$ and an ascent of  type  2
if $\sigma(i)$ is also a left-to-right minimum.  We denote the number of ascents (resp. ascente of type 2 and  left-to-right minima) of $\sigma$ by $\asc\, \sigma$ (resp. $\asc_2\, \sigma$ and $\lrm$). 
A double descent $\sigma(i)$ ($i\in [n]$), i.e., 
$\sigma(i-1)>\sigma(i)>\sigma(i+1)$, 
is called a \emph{foreminimum} of $\sigma$ if it is also a left-to-right minimum. Denote the number of foreminima of $\sigma$ by $\fmin\, \sigma$. It is routine to verify
\begin{align}\label{eq:comp}
(\des_2, \des, \fmax, \rec)\, \sigma=
(\asc_2, \asc, \fmin, \lrm)  \,\sigma^{\rm{c}},
\end{align}
where $\sigma^{\rm{c}}$ is the complementation 
of $\sigma$ (see \eqref{c-r}).

\begin{lemma}\label{des2=pcyc}
There is a bijection $\varphi:\Sym_n\to \Sym_n$ such that
\begin{equation}\label{eq:lemma2.1}
(\des_2, \des, \fmax, \rec)\,\varphi(\sigma)=
(\pcyc, \exc, \fix, \cyc)\,\sigma.
\end{equation}
\end{lemma}
\begin{proof}
We recall a variation of   \emph{Foata's transformation  fondamentale} 
$\phi:\Sym_n\to \Sym_n$, see \cite{FS70} and 
\cite[p.23]{St97}.
Starting from a permutation $\sigma\in\Sym_n$ we factorize it  into
 disjoint cycles, say $\sigma=C_1\, C_2\,\ldots\, C_k$, where
each  cycle $C_i$ is  written  as a sequence
$C_i=(a_i, \sigma(a_i), \ldots, \sigma^{l-1}((a_i))$
with  $\sigma^{l}(a_i)=a_i$ for some $l\in [n]$ and $i\in [k]$. We say that the factorization is
\emph{standard} if
\begin{itemize}
\item  the letter  $a_i$ is the smallest  of the cycle $C_i$,
\item  the sequence $a_1> \cdots> a_k$ is  decreasing.
\end{itemize}
The permutation  $\phi(\sigma)$ is   obtained by dropping
the parentheses in the standard factotorization of $\sigma$.
Note  that  $\asc_2\phi(\sigma)=\pcyc(\sigma)$. Indeed,
  an index $i$ is an ascent of type 2 of $\tau:=\phi(\sigma)$ if  and only if $\tau(i)$ is the smallest element of a pure cycle of $\sigma$.
Thus the transformation  $\phi$ has the following property:
\begin{align}\label{eq:lemma2.1a}
(\asc_2, \asc, \fmin, \lrm)\, \phi(\sigma) =(\pcyc, \exc, \fix, \cyc)\, \sigma.
\end{align}
Let $\varphi(\sigma)=(\phi(\sigma))^c$.
 We  obtain \eqref{eq:lemma2.1} by combining \eqref{eq:lemma2.1a} and \eqref{eq:comp}.
\end{proof}

\begin{rem}
Since $ \des_2=\rec-\fmax$ and $\pcyc=\cyc-\fix$,  the equality
$\rec(\varphi(\sigma))=\cyc(\sigma) $ is redundant in \eqref{eq:lemma2.1}.

\end{rem}
For example, for  $\sigma=23146875\in \Sym_8$, then $\stan(\sigma)=(7)(568)(4)(123)$ and $\phi(\sigma)=75684123$ and $ \varphi(\sigma)=24315876$, it is easy to see that
\begin{align*}
(\des_2, \des, \fmax, \rec)\varphi(\sigma)
&=(\asc_2, \asc, \fmin, \lrm)\, \phi(\sigma) \\
&=(\pcyc, \exc, \fix, \cyc)\, \sigma=(2, 4, 2, 4).
\end{align*}


\begin{proof}[Proof of Theorem~\ref{mainresult:2}]
 The idenitities~\eqref{form1}, \eqref{form2} and \eqref{form3} follow directly from Theorem~\ref{mainresult:1}.  By  Lemma~\ref{des2=pcyc} we derive
Eq.~\eqref{form4}.  Finally, the equivalence of \eqref{eq:zeng93}, \eqref{eq:gfB}  and \eqref{form1} follows from
 \cite[Theorem~1]{Zeng93} and \cite[Theorem~3]{Zeng93}.

\end{proof}
\subsection{Proof of Theorem~\ref{gamma-thm}}

As $D_n(t,\lambda,y)=A_n(t,\lambda, y,0)$, il follows from \eqref{cf:generalA} that 
\begin{align}\label{cf:D}
\sum_{n\geq 0}z^n
D_n(t, \lambda, y)
 \;=\;
 \cfrac{1}
 {1-0z-\cfrac{t\lambda y\,z^2}
 {1-(t+1)z-\cfrac{t(\lambda+1)(y+1)\,z^2}
 {\cdots}}}
 \end{align}
 with $\gamma_n=n(t+1)$ and $\beta_n=t(\lambda+n-1)(y+n-1)$.

Clearly, 
Theorem~\ref{gamma-thm} is proved if we show  that
for each combinatorial interpretation of $\gamma_{n,k}(\lambda, y)$ in  \eqref{gamma1}-\eqref{gamma3}  the corresponding generating function has the following J-fraction expansion
\begin{align}\label{cf:gamma}
\sum_{n\geq 0}z^n\sum_{k\geq 0}\gamma_{n,k}(\lambda, y)t^k
\;=\;
 \cfrac{1}
 {1-0z-\cfrac{t\lambda y\,z^2}
 {1-1\cdot z-\cfrac{t(\lambda+1)(y+1)\,z^2}
 {\cdots}}}
 \end{align}
 with $\gamma_n=n$ and $\beta_n=t(\lambda+n-1)(y+n-1)$.
  \begin{enumerate}[a)]
\item For Eq.\eqref{gamma1}, 
taking the following substitutions in in two sides of \eqref{eq.thm.permutations.Jtype.final1},
\begin{equation}\label{case1bis0}
\begin{cases}
\sfa_{\ell, \ell'}&=t\,(\ell> 0), \quad \sfa_{0, \ell'}=\lambda t;\\
\sfb_{\ell, \ell'}&=1\,(\ell'> 0),\quad \sfb_{\ell, 0}=y;\\
\sfc_{\ell, \ell'}&=1, \quad\sfd_{\ell, \ell'}=0, \quad\sfe_{\ell}=0,
\end{cases}
\end{equation}
we obtain
\begin{align}
\sum_{n\geq 0} z^n \sum_{\sigma\in \SDE_n^*}t^{\exc \sigma} \lambda^{\pex \sigma} y^{\ear\sigma}\;=\;
 \cfrac{1} {1-0z-\cfrac{t\lambda y\,z^2}
 {1-1\cdot z-\cfrac{t(\lambda+1)(y+1)\,z^2}
 {\cdots}}}
 \end{align}
 with $\gamma_n=n$ and $\beta_n=t(\lambda+n-1)(y+n-1)$.

\item For Eq.\eqref{gamma2},
taking the following substitutions in two sides of \eqref{eq.thm.permutations.Jtype.final2},
\begin{equation}
\begin{cases}
\sfa_{\ell}&=t, \quad\sfb_{\ell, \ell'}=1\,(\ell'> 0), \quad\sfb_{\ell, 0}=y;\\
\sfc_{\ell, \ell'}&=1, \quad\sfd_{\ell, \ell'}=0, \quad\sfe_{\ell}=0, 
\end{cases}
\end{equation}
we obtain the smae J-fraction for $\sum_{\sigma\in \SDE_n^*}t^{\exc \sigma} \lambda^{\cyc \sigma} y^{\ear\sigma}$.
\item For Eq.\eqref{gamma3},   we use the dual form 
\eqref{def.Qn.secondmasterdual} for the combinatorial interpretation of
$\widehat{Q}_n(\bsfa,\bsfb,\bsfc,\bsfd,\bsfe,\lambda)$.
Taking the following substitutions in  two sides of \eqref{eq.thm.permutations.Jtype.final2}  
\begin{equation}
\begin{cases}
\sfa_{\ell}&=1, \quad\sfb_{\ell, \ell'}=t\,(\ell'> 0), \quad\sfb_{\ell, 0}=ty;\\
\sfc_{\ell, \ell'}&=0, \quad\sfd_{\ell, \ell'}=1, \quad\sfe_{\ell}=0, 
\end{cases}
\end{equation}
   we obtain the same J-fraction for $\sum_{\sigma\in \SDE_n^*}t^{\exc \sigma} \lambda^{\cyc \sigma} y^{\pex\sigma}$.
\end{enumerate}
\qed

\subsection{Proof of Theorem~\ref{thm1.8}}

Given a permutation $\sigma=\sigma(1)\cdots \sigma(n)$,	
    if $i$ is a descent, i.e., $\sigma(i)>\sigma(i+1)$, 
	 the letters  $\sigma(i)$ and
	$\sigma(i+1)$  are  called {\em descent top} and  {\em descent bottom}, respectively.
For entry $i\in \{1, \ldots, n\}$, we define the
refined patterns, see \cite{SZ12},
\begin{subequations}
\begin{align}\label{31-2}
(31-2)(i,\sigma)&=\#\{j:  1<j<\sigma^{-1}(i) \text{ and } \sigma(j)<i<\sigma(j-1)\};\\
(2-31)(i,\sigma)&=\#\{j:  \sigma^{-1}(i)<j<n \text{ and } \sigma(j+1)<i<\sigma(j)\};\label{2-31}
\end{align}
and  also the refined cycle statistics:
   \begin{align}
\icross(i,\sigma)&=\begin{cases}
\ucross(i,\sigma)&\textrm{if} \; i\in \Cval\sigma\cup\Cdrise\sigma;
\\
\lcross(i,\sigma) &\textrm{if} \; i\in \Cpeak\sigma\cup \Fix\sigma;
\\
\lcross(i,\sigma)+1 &\textrm{if} \; i\in  \Cdfall\sigma.
\end{cases}
\label{connect2}
\\
\cross(i,\sigma)&=
\begin{cases}
\ucross(i,\sigma)&\textrm{if} \; i\in \Cval\sigma\cup \Fix\sigma;
\\
\ucross(i,\sigma)+1 &\textrm{if} \; i\in \Cdrise\sigma;
\\
\lcross(i,\sigma)&\textrm{if} \; i\in \Cpeak\sigma \cup \Cdfall\sigma.
\end{cases}\label{connect1}\\
\nest(i,\sigma)&=\begin{cases}
\unest(i,\sigma)&\textrm{if} \; i<\sigma(i);
\\
\lnest(i,\sigma)&\textrm{if} \; i>\sigma(i);
\\
\lev(i,\sigma) &\textrm{if} \; i=\sigma(i).
\end{cases}\label{connect3}
\end{align}
\end{subequations}
Note that $\ucross(j,\sigma)$ and $\unest(j,\sigma)$ can be nonzero
only when $j$ is a cycle valley or a cycle double rise,
while $\lcross(k,\sigma)$ and $\lnest(k,\sigma)$ can be nonzero
only when $k$ is a cycle peak or a cycle double fall. Thus
$\icross(i,\sigma)=\cross(i,\sigma)=0$ if $i\in \Fix \sigma$.

\subsubsection{The bijection $\Phi_1$}
 Given a permutation $\sigma=\sigma(1)\cdots \sigma(n)$, we proceed as follows:
\begin{itemize} 
\item Determine the sets of descent bottoms $F$  and descent tops $F'$,
 and their complements $G$ and $G'$, respectively;
\item  let  $f$ and  $g$ be  the increasing 
arrangements   of $F$ and  $G$, respectively;
\item construct the biword $f \choose f'$:
for each $j\in f$  starting from the \textbf{largest} (right), the entry in $f'$  below $j$ is the $((31-2)(j,\sigma)+1)$th  largest entry of $F'$  that is
    larger than $j$ and not yet chosen.
    \item construct the biword $g \choose g'$: for each $j$ in $g$ starting from the \textbf{smallest} (left), the entry in $g'$ below $j$
    is the $((31-2)(j,\sigma)+1)$th smallest entry of $G'$ that is not larger than $j$ and not yet chosen;
   
  \item   Form the biword $w=\left({f \atop f'}~{g \atop g'}\right)$ by concatenating
    the biwords ${f\choose g}$, and ${f'\choose g'}$.
    Sorting  the columns so that the top row is in increasing order,
    we obtain the permutation $\tau=\Phi_1(\sigma)$ as the bottom row of the rearranged biword.
\end{itemize}

	\begin{example}\label{ex1}
For  $\sigma=4~7~1~8~6~3~2~5$  we have 
$$
\begin{tabular}{c|c|c|c|c|c|c|c|c}
$\sigma=$&4&7&1&8&6&3&2&5\\
\hline
$(31-2)(i,\sigma)$&0&0&0&0&1&1&1&2\\
\hline
$(2-31)(i,\sigma)$&2&1&0&0&0&0&0&0\\
\end{tabular}\quad
{\Phi_1 \atop \longrightarrow} \quad
\begin{tabular}{c|c|c|c|c|c|c|c|c}
$\tau=$&8&3&6&1&5&7&2&4\\
\hline
$\nest(i,\tau)$   &0&1&1&0&2&0&1&0\\
\hline
$\icross(i,\tau)  $&0&0&0&0&0&1&0&2\\
\end{tabular}
$$
\medskip
\begin{enumerate}
\item Determine the descent bottoms and descent tops  of $\sigma$ and their complements:
\begin{align*}
F&=\{1, 2, 3, 6\}, \quad F'=\{3,6,7,8\},\\
G&=\{4, 5, 7,8\},\quad G'=\{1,2,4,5\}.
\end{align*}
\item Compute the statistics $(31-2)(i,\sigma)$ and $(2-31)(i,\sigma)$ for $i=1,\ldots, 8$, see the above table (left).
\item  Form the biwords, for clarity we write the non-zero $(31-2)(i,\sigma)$ numbers in $f$ and $g$, respectively, as subscripts of their corresponding letters:
$$
	{f\choose f'}
	= \left(
	{1 \atop 8}~
	{2_1 \atop 3}~
	{3_1 \atop 6}~
	{6_1 \atop 7}
	\right),
	\quad
	{g\choose g'}
	= \left(
	{4 \atop 1}~
	{5_2 \atop 5}~
	{7 \atop 2}~
	{8 \atop 4}
	\right).
	$$
\item By concatenation
	$$
	\tau
	= \left( {f \atop f'}~{g \atop g'} \right)
	= \left(
	{1 \atop 8}~
	{2 \atop 3}~
	{3 \atop 6}~
	{4 \atop 1}~
	{5 \atop 5}~
	{6 \atop 7}~
	{7 \atop 2}~
	{8 \atop 4}
	\right).
	$$
	
\end{enumerate}
\end{example}	

\begin{remark}
    We can order  $f'$ and $g'$  in the similar way as in \cite{CSZ97}. Recall that
    an {\em inversion top number} (resp. {\em inversion bottom number}) of a letter $x:=\sigma(i)$
    in the word $\sigma$ is the number of occurrences of inversions of form $(i,j)$ (resp $(j,i)$).
Then  $f'$ (resp. $g'$) is the permutation of descent tops (resp. nondescent tops)  in $\sigma$ such that the inversion bottom (resp. top) number of each
    letter $x$ in $f'$ (resp. $g'$) that below $i$ in $f$ (resp. $g$) is  $(31-2)(i,\sigma)$.
\end{remark}

	\begin{lemma}\label{corteel2}
 Let $\Phi_1(\sigma)=\tau$ for $\sigma\in \Sym_n$. For $i\in \{1, \ldots, n\}$,
 	\begin{subequations}
	   \begin{align}
	   \nest(i,\tau)&=(31-2)(i,\sigma);\label{312-nest}\\
	   \icross(i,\tau)&=(2-31)(i,\sigma).\label{231-icross}
	   \end{align}
	   \end{subequations}
 	\end{lemma}

\begin{lemma}\label{bijection}
The mapping  $\Phi_1$ is a bijection on $\Sym_n$.
\end{lemma}
\begin{proof}  
We construct the inverse $\Phi_1^{-1}$ similarly  as in \cite{CSZ97}.
Let $\Phi_1(\sigma)=\tau$.
Form two biwords $f \choose f'$ and $g \choose g'$, where $f$ (resp. $f'$, $g$, $g'$) is the set of
excedance positions (resp. excedance values, nonexcedance positions, nonexcedance values) of $\tau$, with $f$ and $g$ ordered increasingly and ${i\choose j}$ is a  column if $\tau(i)=j$.  We get $\sigma$ by constructing the descent blocks of $\sigma$, see \cite{CSZ97}.

For $i=1,\ldots, n$, we say that
 $i$ is an opener (resp. closer, insider, outsider) of $\tau$ if $i$ is a letter in  the word $fg'$
 (resp. $f'g,\;  ff,\; gg'$).

If 1 is an outsider, then put 1 as a block. And if 1 is an opener, then put $(\infty,1)$ as an uncomplete block.
For $i\geq 2$, we proceed as follows:
\begin{itemize}
\item  If $i$ is an opener, then put the uncomplete block $(\infty,i)$ to the left of $(\nest(i,\tau)+1)$th uncomplete block from the left.
\item  If $i$ is an outsider, then put $(i)$ as a block to the left of $(\nest(i,\tau)+1)$th uncomplete block from the left.
\item  If $i$ is an insider, then put $i$ into $(\nest(i,\tau)+1)$th uncomplete block from the left and to the right of $\infty$.
\item  if $i$ is a closer,  then replace the $\infty$ of $(\nest(i,\tau)+1)$th uncomplete block by $i$, then the block is complete.
\end{itemize}

After the entry $n$ has been constructed, the blocks are all complete and remove all parenthesis. 
Reading the entries from left to right, we immediately get the permutation $\sigma$.
Let $\Phi_1'(\tau)=\sigma$. Then we prove $\sigma=\sigma$, then $\Phi^{-1}=\Phi_1'$.
The set of descent bottoms (resp. descent tops, nondescent bottoms, nondescent tops) in $\sigma$ transfers to the set of excedance position (resp. excedance value, nonexcedence position, nonexcedence value) in $\tau$ by the operation $\Phi_1$,
and by $\Phi_1'$ it transfer to set of descent top (resp. descent bottom, nondescent top, nondescent bottom) in $\sigma$.
That is the set of descent top, descent bottom, nondescent top and  nondescent bottom in $\sigma$ are the same as those in $\sigma$.
Similarly we get $(31-2)(j,\sigma)=(31-2)(j,\sigma)$ and $(2-31)(j,\sigma)=(2-31)(j,\sigma)$ for $j\in[n]$.

Suppose $\sigma(j)=\sigma(j)$ for $j=1,2,\ldots,i-1$, and $\sigma(i)\neq \sigma(i)$.
Assume $\sigma(i)=k<l=\sigma(i)$.
If $\sigma(i-1)$ is a descent top, then $k,l$ are descent bottoms.
Then for $l$ in $\sigma$ and $\sigma$, since $\sigma(i-1)>k>l$, we have $(31-2)(l,\sigma)>(31-2)(l,\sigma)$, contradictory to $(31-2)(l,\sigma)=(31-2)(l,\sigma)$.
If $\sigma(i-1)$ is a nondescent top, then $k,l$ are nondescent bottoms.
Suppose $\sigma(j)=k$, then $\sigma(j-1)<k$ since $k$ is a nondescent bottom. And since $l>k$, then there exist two consecutive index $i<x,x+1<j$ such that $\sigma(x)>k>\sigma(x+1)$.
Then we have $(31-2)(k,\sigma)<(31-2)(k,\sigma)$, contradictory to $(31-2)(k,\sigma)=(31-2)(k,\sigma)$.
We complete the proof.
\end{proof}

\begin{example}
We illustrate $\Phi_1^{-1}$ on $\tau=8~3~6~1~5~7~2~4$.
\begin{align*}
(\infty,1)&\longrightarrow(\infty,1)(\infty,2)\\
\longrightarrow(\infty,1)(\infty,3,2)&\longrightarrow (4)(\infty,1)(\infty,3,2)\\
\longrightarrow(4)(\infty,1)(\infty,3,2)(5)&\longrightarrow(4)(\infty,1)(\infty,6,3,2)(5)\\
\longrightarrow(4)(7,1)(\infty,6,3,2)(5)&\longrightarrow(4)(7,1)(8,6,3,2)(5).
\end{align*}
Then we have $\Phi_1^{-1}(\tau)=\sigma=4~7~1~8~6~3~2~5$.
\end{example}

\begin{lemma}\label{lem:exc-drop2}
		Let $\Phi_1(\sigma)=\tau$ for $\sigma\in \Sym_n$. Then
	\begin{align}\label{exc-des}
 		( \cpeak,\cval, \cdrise, \cdfall+\fix)\,\tau=( \peak, \valley, \ddes, \dasc)\,\sigma,
 		\end{align}
		\begin{align}\label{mapping: phi1}
	(\des, \des_2)	\sigma=(\exc, \ear) \Phi_1(\sigma).
		\end{align}
	\end{lemma}

\begin{proof}

We just prove \eqref{exc-des}.
  If $j$ is a peak of $\sigma$, then $j$ is a descent top and nondescent bottom, so $j$ is in $f'$ and $g$. Since in $f \choose f'$, $\tau(k)=j>k$, and in $g \choose g'$, $\tau(j)=l <j$, then $j$ is a cycle peak in $\tau$.
If $j$ is a valley  in $\sigma$, then  $j$ is a descent bottom of $\sigma$ and nondescent top of $\sigma$,
so $j$ is in $f$ and $g'$, so $j$ is a cycle valley of $\tau$. 
 Similarly we have if $j$ is a double descent of $\sigma$, then $j$ is in $f$ and $f'$, so $j$ is a cycle double rise in $\tau$.
  And if $j$ is a double ascent of $\sigma$, then $j$ is in $g$ and $g'$, so $j$ is a cycle double fall or a fixed point in $\tau$.
Let $\Phi_1(\sigma)=\tau$. 
It is not difficult to check that $\des=\peak+\ddscent=\valley+\ddscent$ and by Lemma \ref{lem:pexear}, $\exc=\cval+\cdrise$. 
So by Lemma \ref{corteel2} we have $\exc(\tau)=\des(\sigma)$.
If $i$ is a left-to-right maximum of $\sigma$, then $(31-2)(i,\sigma)=0$.
By definition of descent of type 2, $i\in \Des_2(\sigma)$ if and only if $i$ is a descent and left to right maximum,
then $i$ is a peak and $(31-2)(i,\sigma)=0$.
By Lemma \ref{corteel2} we have $i\in \Cpeak(\tau)$ and $\lnest(i,\tau)=0$, that is $i$ is an eareccpeak.
If $i\notin \Des_2(\sigma)$, then either $i$ is not a peak or $(31-2)(i,\sigma)\neq 0$.
Then we have $i\notin \Cpeak(\tau)$ or $\lnest(i,\tau)\neq 0$, that is $i$ is not an eareccpeak.
So we have $\ear(\tau)=\des_2(\sigma)$.
\end{proof}

\subsubsection{The bijection $\Phi_2$}
We first recall the bijection  $\Phi_{SZ}$ of Shin and Zeng~\cite{SZ10}, which is a 
variation of the bijection $\Phi$ in \cite{CSZ97}.
    Given a permutation $\sigma=\sigma(1)\cdots \sigma(n)$, we proceed as follows:
    \begin{itemize}
    \item Determine the sets of descent tops $F$  and descent bottoms $F'$,
 and their complements $G$ and $G'$, respectively;
    \item let $f$ and  $g$ be  the  increasing permutations of $F$ and $G$, respectively;
\item construct the biword $f \choose f'$: for each $j$ in the first row $f$ starting from the \textbf{smallest} (left), the entry in $f'$ that below $j$ is the $((31-2)(j,\sigma)+1)$th
    largest entry of $f'$ that is smaller than $j$ and not yet chosen.
    \item construct the biword $g \choose g'$: for each $j$ in the first row  $g$ starting from the \textbf{largest} (right), the entry  below $j$ in $g'$ 
    is the $((31-2)(j,\sigma)+1)$th smallest entry of $g'$ that is not smaller than $j$ and not yet chosen.
    \item
      Rearranging the columns so that the top row is in increasing order,
    we obtain the permutation $\tau=\Phi_{SZ}(\sigma)$ as the bottom row of the rearranged biword.
\end{itemize}


	\begin{example}\label{ex2}
For  $\sigma=4~7~1~8~6~3~2~5$ with
$$
\begin{tabular}{c|c|c|c|c|c|c|c|c}
$\sigma=$&4&7&1&8&6&3&2&5\\
\hline
$(31-2)(i,\sigma)$&0&0&0&0&1&1&1&2\\
\hline
$(2-31)(i,\sigma)$&2&1&0&0&0&0&0&0\\
\end{tabular}
\quad
{\Phi_{SZ} \atop \longrightarrow} \quad
\begin{tabular}{c|c|c|c|c|c|c|c|c}
$\tau=$&5&7&1&4&8&2&6&3\\
\hline
$\cross(i,\tau)$&2&0&0&0&0&1&1&1\\
\hline
$\nest(i,\tau)$&0&1&0&2&0&0&0&0\\
\end{tabular}.
$$
We have
$$
	{f\choose f'}
	= \left(
	{3_1 \atop 1}~
	{6_1 \atop 2}~
	{7 \atop 6}~
	{8 \atop 3}
	\right),
	\quad
	{g\choose g'}
	= \left(
	{1 \atop 5}~
	{2_1 \atop 7}~
	{4 \atop 4}~
	{5_2 \atop 8}
	\right).
	$$
	Hence
	$$
	\tau
	= \left( {f \atop f'}~{g \atop g'} \right)
	= \left(
	{1 \atop 5}~
	{2 \atop 7}~
	{3 \atop 1}~
	{4 \atop 4}~
	{5 \atop 8}~
	{6 \atop 2}~
	{7 \atop 6}~
	{8 \atop 3}
	\right).
	$$
\end{example}		


Similar to Lemma \ref{corteel2} and Lemma \ref{bijection}  we have
the following result in \cite{SZ10}.

 \begin{lemma}\label{corteel} The mapping $\Phi_{SZ}: \Sym_n\to \Sym_n$ is a bijection.
 	For $\sigma\in \Sym_n$, if $\Phi_{SZ}(\sigma)=\tau$, then
\begin{subequations}
for $i\in [n],$
       \begin{align}
	   (31-2)(i,\sigma)=\cross(i,\tau),\\
	   (2-31)(i,\sigma)=\nest(i,\tau).
	   \end{align}
	   \end{subequations}
		\end{lemma}
	
	By Lemma \ref{lem:pexear}, if $i\in \Pex(\sigma)$, then $i\in \Cval(\sigma)$ and $\ucross(i,\sigma)=0.$
    And by definition of "$\pdrop$", we see that
    if  $i\in \Pdrop(\sigma)$, then $i\in \Cpeak(\sigma)$ and
    $\lcross(i,\sigma)=0.$ 
Recall that   Fix (resp. Valley, Peak, Ddes, Dasc, Fmax) is the set valued 
statistic, namely ${\Fix}\,\sigma$ denotes the set of fixed points of $\sigma$.

	\begin{lemma}\label{kappa}
		We have
			$(\des, \des_2, \fmax)\sigma=(\drop, \pdrop, \fix) \Phi_{SZ}(\sigma)$.
	\end{lemma}
\begin{proof}
Let $\Phi_{SZ}(\sigma)=\tau$.
Then similar to the proof of Lemma \ref{lem:exc-drop2},
we prove that
\begin{align}\label{des-drop}
(\Cval, \Cpeak, \Cdfall, \Cdrise+\Fix, \Fix)\,\tau=(\Valley, \Peak,  \Ddes, \Dasc, \Fmax)\,\sigma.
\end{align}
Note that if $i$ is a foremaximum, then $i$ is a double ascent and $(31-2)(i,\sigma)=0$.
Then by the bijection $\Phi_{SZ}$,  $i$ is in $g$ and $g'$.
Since in $g$ the entry is sorted from the largest, for all the entries $j>i$ in $g$, we have $\tau(j)\geq j>i$.
We obtain that when $i$ is sorted, $i$ is not chosen and the smallest entry that are not chosen is $i$.
By $(31-2)(i,\sigma)=0$, we get $i$ is a fixed point.
It is not difficult to check that $\des=\peak+\ddscent$ and $\drop=\cpeak+\cdfall$.
By \eqref{des-drop} we have $\drop(\tau)=\des(\sigma)$.
If $i$ is a descent of type 2, then $i$ is a peak and $(31-2)(i,\sigma)=0$.
By \eqref{des-drop} and Lemma \ref{corteel} we have $i\in \Cpeak(\tau)$ and $\lcross(i,\tau)=0$, that is $i$ is a pure drop.
If $i$ is not a descent of type 2, then we have $i\notin \Cpeak(\tau)$ or $\lcross(i,\tau)\neq 0$, that is $i$ is not a pure drop.
So we have $\pdrop(\tau)=\des_2(\sigma)$.
\end{proof}

Let $\Phi_2=\zeta\circ \Phi_{SZ}$, where 
$\zeta$ is  the \emph{reversal and complementation} operation
	in \eqref{zeta}.
Combining Lemma \ref{kappa} and Lemma \ref{lem:exc-drop1} we obtain Eq. \eqref{des2-pex}. An illustration for the permutation $\tau$ in Example \ref{ex2} is given in 
Figure \ref{fig.pictorial}. 
\section{Interpretations using linear statistics}

The polynomials $Q_n(\bsfa,\bsfb,\bsfc,\bsfd,\bsfe)$  being  defined using cyclic statistics in \eqref{def.Qn.firstmaster},
by the bijection $\Phi_{SZ}$, Lemma \ref{corteel} and Eq. \eqref{des-drop}, 
we derive  the following  interpretation  using linear statistics.
\begin{theorem}(First linear version of $Q_n$)
\begin{eqnarray}\label{eq:linearQ1} 
& & \hspace*{-10mm}
Q_n(\bsfa,\bsfb,\bsfc,\bsfd,\bsfe)
\;=\;
\nonumber \\[4mm]
& &
\sum_{\sigma \in \Sym_n}
\;\:
\prod\limits_{i \in {\rm \Valley}}  \! \sfa_{(31-2)(i,\sigma),\,(2-31)(i,\sigma)}
\prod\limits_{i \in {\rm \Peak}} \!\!  \sfb_{(31-2)(i,\sigma),\,(2-31)(i,\sigma)}
\:\times
\qquad\qquad
\nonumber \\[1mm]
& & \qquad\;
\prod\limits_{i \in {\rm \Ddes}} \!\!  \sfc_{(31-2)(i,\sigma),\,(2-31)(i,\sigma)}
\;
\prod\limits_{i \in {\rm \Dasc\setminus\Fmax}} \!\!  \sfd_{(31-2)(i,\sigma)-1,\,(2-31)(i,\sigma)}
\, \prod\limits_{i \in {\rm \Fmax}} \sfe_{(2-31)(i,\sigma)}
\;.
\quad
\end{eqnarray}

\end{theorem}

 For $\sigma\in\Sym_n$,  a value $\sigma(i)$ is an 
\emph{antirecord double ascent} 
($\textrm{Arda}$)
if it is an antirecord and at the same time a double ascent.
Applying the bijection $\Phi_1$ we obtain another interpretation of the 
polynomials $Q_n$.
\begin{theorem}(Second linear version of $Q_n$)
\begin{eqnarray}
& & \hspace*{-10mm}
Q_n(\bsfa,\bsfb,\bsfc,\bsfd,\bsfe)
\;=\;
\nonumber \\[4mm]
& &
\sum_{\sigma \in \Sym_n}
\;\:
\prod\limits_{i \in {\rm \Valley}}  \! \sfa_{(2-31)(i,\sigma),\,(31-2)(i,\sigma)}
\prod\limits_{i \in {\rm \Peak}} \!\!  \sfb_{(2-31)(i,\sigma),\,(31-2)(i,\sigma)}
\:\times
\qquad\qquad
\nonumber \\[1mm]
& & \qquad\;
\prod\limits_{i \in {\rm \Dasc \setminus \Arda}} \!\!  \sfc_{(2-31)(i,\sigma)-1,\,(31-2)(i,\sigma)}
\;
\prod\limits_{i \in {\rm \Ddes}} \!\!  \sfd_{(2-31)(i,\sigma),\,(31-2)(i,\sigma)}
\, \prod\limits_{i \in \Arda} \sfe_{(31-2)(i,\sigma)}
\;.
\quad
\end{eqnarray}

\end{theorem}
\begin{proof} In view of Lemma~\ref{corteel} we just need to show that 
$$\Arda\; \sigma= \Fix \Phi_1(\sigma).$$
	  If $i$ is an antirecord, then $(2-31)(i,\sigma)=0$.
	 Thus if $i$ is a $\Arda$, then $i$ is a double ascent and $(2-31)(i,\sigma)=0$.
	Then by the bijection $\Phi_1$,  $i$ is in $g$ and $g'$.
	Since in $g$ the entry is sorted from the smallest, for all the entries $j<i$ in $g$, we have $\tau(j)\leq j<i$.
	We obtain that when $i$ is sorted, $i$ is not chosen and the  largest entry that are not chosen is $i$.
	By $(2-31)(i,\sigma)=0$, we get $i$ is a fixed point.
    Suppose $i$ is not a ARL, if $i$ is not a double ascent,
    then $i$ is either not in $g$ or in $g'$.
    So $i$ is not a fixed point. 
    And if $i$ is a double ascent,
    since $(2-31)(i,\sigma)\neq 0$ and $i$ is the largest entry that are not chosen, then $i$ is not a fixed point.
\end{proof}

\begin{rem}
Using the reversal transformation  $\sigma\mapsto \sigma^r$, we obtain  a dual version of the above theorems for the boundary
 condition $\infty-0$, i.e., $\sigma(0)=n+1$ and $\sigma(n+1)=0$.
\end{rem}

As an application of  Theorem~\ref{eq:linearQ1} we give  a linear version for   
$A_n(t, \lambda, y, w)$ and its  derangement counterpart $D_n(t, \lambda, y)$.
This enables us to give a group action proof of  \eqref{Equ1} in Theorem~\ref{gamma-thm}. 

For $\sigma\in \Sym_n$,  recall that $\asc\,\sigma$ is the number of ascents of $\sigma$ (cf. \eqref{eq:comp}) and  define the statistics \emph{pure valley} ($\pval$) and 
\emph{pure peak} ($\ppk$) by
\begin{align}
\pval\,\sigma =&|\{i\in[n]\colon i \in\Valley\,\sigma\,\, \text{and}\,\,\, \cab(i , \sigma)=0\}|,\label{def:pval}\\
\ppk\,\sigma=&|\{i\in[n]\colon i  \in\Peak\,\sigma\,\, \text{and}\,\,\, \bca(i , \sigma)=0\}|.\label{def:ppk}
\end{align}

\begin{lemma}
We have 
\begin{equation}\label{linear:A}
A_n(t, \lambda, y, w)=
\sum_{\sigma\in\Sym_n}t^{(\asc-\fmax)\,\sigma}\lambda^{\pval\,\sigma}y^{\ppk\,\sigma}w^{\fmax\,\sigma}.
\end{equation}
\end{lemma}
\begin{proof}
Applying the substitution  \eqref{case1} in Eq.~\eqref{eq:linearQ1} we obtain the right-hand side in \eqref{linear:A}  for
$Q_n(\bsfa,\bsfb,\bsfc,\bsfd,\bsfe)$, which is equal to 
$A_n(t, \lambda, y, w)$ by \eqref{Equ1}.
\end{proof}

Let $\Sym_{n,j}:=\{\sigma\in\Sym_n\colon \fmax(\sigma)=j\}$,  
${\Sym}^*_{n,j}:=\{\sigma\in \Sym_{n,j}\colon  \dd(\sigma)=0\}$
and
\begin{align}
{\Sym}^*_{n,j}(k):=\{\sigma\in \Sym^*_{n,j}\colon  \des(\sigma)=k\}.
\end{align}

\begin{theorem}\label{thm:gamma4} We have 
\begin{equation}\label{D-linearversion}
D_n(t, \lambda, y)=
\sum_{\sigma\in \Sym_{n,0}}t^{\asc\,\sigma}\lambda^{\pval\,\sigma}y^{\ppk\,\sigma}
\end{equation}
and
\begin{equation}\label{gamma-LD}
D_n(t, \lambda, y)=\sum_{k=0}^{\left\lfloor n/2\right\rfloor}\gamma_{n,k}(\lambda, y)
t^k(1+t)^{n-2k},
\end{equation}
where the gamma coefficient $\gamma_{n,k}(\lambda, y)$ has the 
following combinatorial interpretation
\begin{align}
\gamma_{n,k}(\lambda, y)=\sum_{\sigma\in{\Sym}^*_{n,0}(k)}\lambda^{\pval\,\sigma}y^{\ppk\,\sigma}.\label{gamma4}
\end{align}
\end{theorem}
\begin{proof}
As $D_n(t, \lambda, y)=A_n(t, \lambda, y, 0)$,  letting $w=0$ in \eqref{linear:A} we obtain \eqref{D-linearversion}.
Taking the following substitutions 
\begin{equation}\label{case1bis}
\begin{cases}
\sfa_{\ell, \ell'}&=t\,(\ell> 0), \quad \sfa_{0, \ell'}=\lambda t;\\
\sfb_{\ell, \ell'}&=1\,(\ell'> 0),\quad \sfb_{\ell, 0}=y;\\
\sfc_{\ell, \ell'}&=0, \quad\sfd_{\ell, \ell'}=t, \quad\sfe_{\ell}=0,
\end{cases}
\end{equation}
in \eqref{def.Qn.firstmaster} and \eqref{eq:linearQ1}, respectively, 
we obtain
\begin{equation}\label{gamma:cyc-lin}
\sum_{\sigma\in\SDE^*_n}t^{\exc\,\sigma}\lambda^{\pex\,\sigma}y^{\eareccpeak\,\sigma}=
\sum_{\sigma \in {\Sym}^*_{n,0}}t^{\asc\,\sigma}\lambda^{\pval\,\sigma}y^{\ppk\,\sigma}.
\end{equation}
Extracting  the coefficient of $t^{k}$ in \eqref{gamma:cyc-lin}  we deduce \eqref{gamma-LD} from \eqref{Equ1}.
\end{proof}

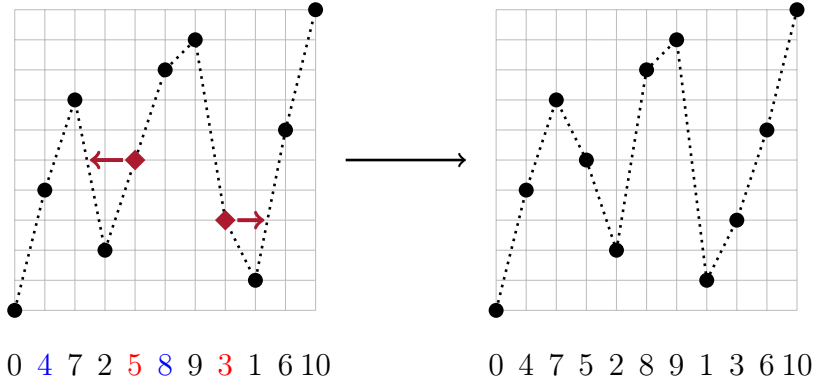
\begin{figure}[t]
\begin{center}
\begin{tikzpicture}[scale=0.4] 	
\draw[step=1,lightgray,thin] (0,0) grid (10,10); 
	\tikzstyle{ridge}=[draw, line width=1, dotted, color=black] 
	\path[ridge] (0,0)--(1,4)--(2,7)--(3,2)--(4,5)--(5,8)--(6,9)--(7,3)--(8,1)--(9,6)--(10,10); 
	\tikzstyle{node0}=[circle, inner sep=2, fill=black] 
	\tikzstyle{node1}=[rectangle, inner sep=3, fill=mhcblue] 
	\tikzstyle{node2}=[diamond, inner sep=2, fill=davidsonred] 
	\node[node0] at (0,0) {}; 
	\node[node0] at (1,4) {}; 
	\node[node0] at (2,7) {}; 
	\node[node0] at (3,2) {}; 
	\node[node2] at (4,5) {}; 
	\node[node0] at (5,8) {}; 
	\node[node0] at (6,9) {}; 
	\node[node2] at (7,3) {}; 
	\node[node0] at (8,1) {};
	\node[node0] at (9,6){};
	\node[node0] at(10,10){}; 
	\tikzstyle{hop1}=[draw, line width = 1.5, color=davidsonred,->]
	\tikzstyle{hop2}=[draw, line width = 1.5, color=davidsonred,<-] 
	\path[hop2] (2.5,5)--(3.6,5);
	\path[hop1] (7.4,3)--(8.3, 3);
	\tikzstyle{pi}=[above=-1] 
	\node[pi] at (0,0) {0}; 
	\node[pi] at (1,0) {$\blue{4}$}; 
	\node[pi] at (2,0) {$7$}; 
	\node[pi] at (3,0) {2}; 
	\node[pi] at (4,0) {\red{5}}; 
	\node[pi, color=davidsonred] at (5,0) {$\blue{8}$}; 
	\node[pi] at (6,0) {$9$}; 
	\node[pi] at (7,0) {\red{3}}; 
	\node[pi] at (8,0) {1};
	\node[pi] at (9,0) {6};
	\node[pi] at (10,0) {10};
	\path[draw,line width=1,->] (11,5)--(15,5); 
	\begin{scope}[shift={(16,0)}] 
	\draw[step=1,lightgray,thin] (0,0) grid (10,10); 
	\path[ridge] (0,0)--(1,4)--(2,7)--(3,5)--(4,2)--(5,8)--(6,9)--(7,1)--(8,3)--(9,6)--(10,10); 
	\node[node0] at (0,0) {}; 
	\node[node0] at (1,4) {}; 
	\node[node0] at (2,7) {}; 
	\node[node0] at (3,5) {}; 
	\node[node0] at (4,2) {}; 
	\node[node0] at (5,8) {}; 
	\node[node0] at (6,9) {}; 
	\node[node0] at (7,1) {}; 
	\node[node0] at (8,3) {}; 
	\node[node0] at (9,6){};
	\node[node0] at (10,10){};
        \node[pi] at (0,0) {0}; 
	\node[pi] at (1,0) {$4$}; 
	\node[pi] at (2,0) {$7$}; 
	\node[pi] at (3,0) {5}; 
	\node[pi] at (4,0) {2}; 
	\node[pi] at (5,0) {$8$}; 
	\node[pi] at (6,0) {$9$};
	\node[pi] at (7,0) {1}; 
	\node[pi] at (8,0) {3};
	\node[pi] at (9,0){6};
	\node[pi]  at (10,0){10};
\end{scope}
\end{tikzpicture}
\end{center}
\caption{Valley-hopping on $\sigma=472589316\in \Sym_9$ with $S=\{3, 4, 5\}$
yields $\varphi'_{S}(\sigma)=475289136$\label{valhopFmax}}
\end{figure}

\subsection{Group action proof of \eqref{gamma-LD}}

In the following we give a direct proof of \eqref{gamma-LD} by applying  the  well-known  \textit{valley-hopping} action, see 
Foata and Strehl
\cite{FS74},  Shapiro, Woan, and Getu \cite{SWG83}
and  Br\"and\'en \cite{Bra08}.
Let $\sigma\in\Sym_n$ with boundary condition $\sigma(0)=0$ and $\sigma(n+1)=n+1$.
 Recall that for $x\in[n]$, see \cite{FS74}, 
 the {\em$x$-factorization} of $\sigma$ is defined by
\begin{align}\label{x-factorization}
 \sigma=w_1 w_2x w_3 w_4,
\end{align}
 where $w_2$ (resp.~$w_3$) is the maximal contiguous subword immediately to the left (resp.~right) of $x$ whose letters are all less than $x$. 
 Note that $w_1,\ldots, w_4$ may be empty. For instance,   
 if $x$ is a double ascent (resp. double descent), then $w_3=\varnothing$ (resp. $w_2=\varnothing$), and if $x$ is a valley then
$w_2=w_3=\varnothing$.
Foata and Strehl~\cite{FS74} considered a mapping $\varphi_x$ on permutations by
exchanging $w_2$ and $w_3$ in \eqref{x-factorization}:
$$
\varphi_x(\sigma)=w_1 w_3x w_2 w_4.
$$
For instance, if $x=3$ and $\sigma=472589316\in\Sym_9$, then $w_1=472589,w_2=\varnothing,w_3=1$ and $w_4=6$.
Thus $\varphi_3(\sigma)=472589136$.
It is known (see \cite{FS74}) that  $\varphi_x$ is an involution acting on $\Sym_n$ and  that $\varphi_x$ and $\varphi_y$ commute for all $x,y\in[n]$. Br\"and\'en~\cite{Bra08} introduced the modified mapping  $\varphi'_x$ by
\begin{align*}
\varphi'_x(\sigma):=
\begin{cases}
\varphi_x(\sigma),&\text{if $x$ is neither a  peak  nor foremaximum  of $\sigma$ 
};\\
\sigma,& \text{if $x$ is a peak or foremaximum  of $\sigma$.}
\end{cases}
\end{align*}
Note that the boundary condition matters, e.g.,  in the above example, 
if $\sigma(0)=10$ instead, then $4$ becomes a valley and will be fixed by $\varphi'_4$.
Also, we have $\varphi'_{x}(\sigma)=\sigma$ if $x$ is a peak, valley or foremaximum, otherwise $\varphi'_{x}(\sigma)$ exchanges $w_2$ and $w_3$ in the $x$-factorization of $\sigma$, which is equivalent to moving $x$ from a double ascent to a double descent or vice versa. Then $\varphi'_{x}$'s are involutions and commute.  Hence,  
for any subset $S\subseteq[n]$ we can define the map $\varphi'_S :\Sym_n\rightarrow\Sym_n$ by
\begin{align*}
\varphi'_S(\sigma)=\prod_{x\in S}\varphi'_x(\sigma).
\end{align*}
In other words,  the group $\Z_2^n$ acts on $\Sym_n$ via the mapping $\varphi'_S$ with $S\subseteq[n]$. 
For example,  let  $\sigma=472589316\in \Sym_9$, then $\Fmax(\sigma)=\{4, 8\}$.
If  $S=\{3, 4,5\}$, we have $\varphi'_{S}(\sigma)=475289136$, 
see Fig.~\ref{valhopFmax} for an illustration.

Recall  that  a permutation $\sigma$ has a descent at $i$ with $1\leq i<n$ if
 $\sigma(i)>\sigma(i+1)$.
We will say that a \emph{run} of a permutation $\sigma$ is a maximal interval of consecutive arguments of $\sigma$  on which the values of $\sigma$ are monotonic. If the values of $\sigma$ increase on the interval then we speak of a rising run, else a descreasing run.

\begin{lemma}\label{lem:231312}
For $\sigma\in \Sym_n$ 
 the quintuple permutation statistic $(\peak, \valley, \fmax, \ppk, \pval)\sigma$ is invariant under the group action $\varphi'_S$ with  $S\subset [n]$, i.e.,
 \begin{align}\label{quintuple}
 (\peak, \valley, \fmax, \ppk, \pval)\;\sigma=(\peak, \valley, \fmax, \ppk, \pval)\;\varphi'_S(\sigma).
\end{align}
\end{lemma}
\begin{proof}
Clearly  the triple statistic $(\peak, \valley, \fmax)$ is invariant under the valley-hopping action, so it remains to verify the invariance for bi-statistic $(\ppk, \pval)$.
For  each $i\in[n]$, the statistic
 $\bca(i,\sigma)$ equals the number of maximal decreasing runs at the right of $i$, i.e., the sequences of values with consecutives arguments 
 $k, k+1, \ldots, j$ such that $\sigma^{-1}(i)<k<j$,
 $$
 \sigma(k)>i>\sigma(j)\qquad 
 \textrm{with $\sigma(k)\in {\rm Peak}(\sigma)$ and $\sigma(j)\in {\rm Valley}(\sigma)$},
 $$
 see Fig.\ref{fig.4}.
Since  the elements of $\textrm{Peak}(\sigma)$ and $\textrm{Valley}(\sigma)$ are fixed  by the valley-hopping action,  the number 
$\bca(i, \sigma)$ is invariant under the group action if 
$i\in\Peak(\sigma)$.  This proves the invariance of $\ppk$ in \eqref{def:ppk}.
The case  of $\pval$ is similar and omitted.
\end{proof}

\begin{figure}[t]
\begin{center}
\setlength {\unitlength} {0.8mm}
\begin {picture} (70,45) \setlength {\unitlength} {1mm}
\thinlines





%

\put(-37, -2){$0$}
\put(-36, 2){\circle*{1.3}}
\put(-36,2){\line(1,1){18}}
\put(-18, 20){\circle*{1.3}}
\put(-18,20){\line(1,-1){15}}
\put(-23,23){$\sigma(i)$}
\put(-3,5){\circle*{1.3}}
\put(-3,5){\line(1,1){21}}
\put(15,28){$\sigma(k)$}
\put(18, 26){\circle*{1.3}}

\put(18,26){\line(1,-1){24}}
\put(40,-2){$\sigma(j)$}
\put(42,2){\circle*{1.3}}
\put(42,2){\line(1,1){10}}
\put(52,12){\circle*{1.3}}
\put(52,12){\line(1,-1){7}}

\put(59,5){\line(1,1){26}}
\put(59,5){\circle*{1.3}}
\put(85,32){$\infty$}


\end{picture}
\vspace{10pt}
\end{center}
\caption{\label{fig.4}Valley-hopping on statistics $\ppk$
\label{231}}
\end {figure}
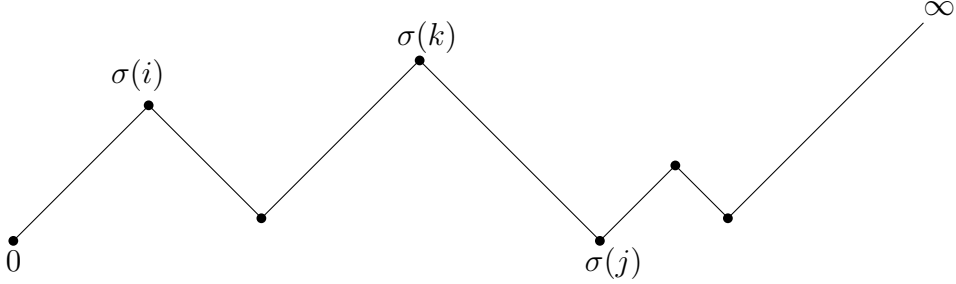

\begin{lemma} We have 
\begin{align}\label{eq:fmaxgamma}
\sum_{\sigma\in \Sym_{n,j}}\lambda^{\pval\,\sigma}y^{\ppk\,\sigma} t^{(\asc-\fmax) \sigma} = 
 \sum_{k=0}^{\lfloor (n-j)/2 \rfloor} 
 \Biggl(\sum_{\sigma\in \Sym^*_{n,j}(k)}\lambda^{\pval\,\sigma}y^{\ppk\,\sigma}\Biggr)\, t^k (1+t)^{n-j-2k}.
\end{align}
\end{lemma}
\begin{proof}
For any  permutation $\sigma\in \Sym_n$,  let $\Orb(\sigma)=\{g(\sigma): g\in\Z_2^n\}$ be the orbit of $\sigma$ under the valley-hopping. 
The valley-hopping divides the set $\Sym_n$ into disjoint orbits. 
 Moreover, for $\sigma\in \Sym_n$, if $x$ is a double descent of $\sigma$, then 
 $x$ is a double ascent of $\varphi'_x(\sigma)$.
 Hence, there is a unique  
  permutation in each orbit which has no 
 double descent.
 Now, let 
  $\bar{\sigma}$ be such a unique element  in $\Orb(\sigma)$, then 
  \begin{align*}
  \dasc(\bar{\sigma})&=n-\peak(\bar{\sigma})-\valley(\bar{\sigma});\\
   \des(\bar{\sigma})&=\peak(\bar{\sigma})=\valley(\bar{\sigma}).
   \end{align*}
As $\asc-\fmax=valley+\dasc-\fmax$, we have 
\begin{align}\label{orbit}
\sum_{\sigma'\in\Orb(\sigma)}t^{(\asc-\fmax)\,\sigma'}=
t^{\valley(\bar{\sigma})}(1+t)^{(\dasc-\fmax)(\bar{\sigma})}
=t^{\des(\bar{\sigma})}(1+t)^{n-2\des \sigma -\fmax\,(\bar{\sigma})}.
\end{align}
Therefore, by Lemma~\ref{lem:231312}, we obtain \eqref{eq:fmaxgamma}.
\end{proof}

Clearly \eqref{gamma-LD} corresponds to the $j=0$ case of \eqref{eq:fmaxgamma}.

\section{Concluding remarks}
We notice  that the bistatistics $(\des_2, \fix)$ and $(\pex, \fix)$  are not equidistrubuted on $\Sym_4$ and  the distribution of 
$(\des_2, \pex)$  over $\Sym_6$ is not symmetric.  Let
\begin{align}
P_n(x,y)=\sum_{\sigma \in \Sym_n} x^{\des_2(\sigma)}y^{\ear(\sigma)}.
\end{align}  
We speculate that  the polynomial $P_n(x,y)$ is invariant under 
$x\leftrightarrow y$. 
\begin{conj}  The distribution of $(\des_2, \ear)$
  over permutations is symmetric.
\end{conj}

 By Theorem~\ref{mainresult:3} and Eq. \eqref{cf:generalA},  we can reformulate 
Conjecture~\ref{conj:1} as follows.

\begin{conjecture}
\begin{align}
\sum_{n\geq 0}\sum_{\sigma\in \Sym_n} y^{\des_2\sigma} \lambda^{\cyc\sigma} \,z^n
=\frac{1}{1-\lambda z-\cfrac{\lambda yz^2}{1-(\lambda+2)\, z-
\cfrac{(\lambda+1)(y+1) z^2}{\cdots}}}
\end{align}
with $\gamma_n=\lambda+2n$ and $\lambda_n=(\lambda+n-1)(y+n-1)$.
\end{conjecture}

The Fran\c con-Viennot bijection $\Phi_{FV}$  and Foata-Zeilberger bijection
$\Phi_{FZ}$
 are two fundamental  bijections from permutations to \emph{Laguerre histories} \cite{FV79,FZ90}. The composition $\Phi_{FZ}^{-1}\circ \Phi_{FV}$  as a bijection $\Phi$ on $\Sym_n$
 was first characterized in \cite{CSZ97}. Since then
 some variations of this bijection appeared in \cite{Cor06, SZ10, SZ12, HMZ20}. 
Our bijections $\Phi_1$ and $\Phi_2$ are similar to  
those in \cite{Cor06, SZ10}. 
Besides,  the two equations \eqref{form2} and \eqref{form4} ask for  a bijection $\Phi_3$ on $\Sym_n$ satisfying 
\begin{equation}
(\des, \des_2, \fmax)\,\sigma
=(\exc, \ear,\fix)\,\Phi_3 (\sigma).\label{des2-ear-fix}
\end{equation}
Although a bijection via Laguerre histories could be given by combining   $\Phi_{FV}$  
and 
$\Phi_{FZ}^{-1}$,  a direct bijection 
similar to $\Phi_1$ and  $\Phi_2$  would be interesting.

\section*{Acknowledgement} We are grateful to the anonymous referee for 
his/her careful reading and constructive suggestions.


\end{document}